\documentclass{article}

\usepackage{amsmath,amssymb}	% to deal with mathematics

\usepackage{graphicx}		  % to deal with graphicals

\newcommand{\Ci}{\mathscr{C}}
\newcommand{\Sone}{\mathbb{S}^1}

\newcommand{\N}{\mathbb{N}}
\newcommand{\R}{\mathbb{R}}

\newcommand{\Z}{\mathbb{Z}}

\newcommand{\vertiii}[1]{{\left\vert\kern-0.25ex\left\vert\kern-0.25ex\left\vert #1 
    \right\vert\kern-0.25ex\right\vert\kern-0.25ex\right\vert}}

\usepackage[initials]{amsrefs}
\usepackage{amsmath,amsthm,amscd}
\usepackage{enumerate}

\usepackage{amssymb}
\usepackage{mathrsfs}
\usepackage{graphicx}

\usepackage{color}
\usepackage{mathtools}
\mathtoolsset{showonlyrefs=true}

\usepackage[T1]{fontenc}

\usepackage{amsmath,amsthm,amssymb,amscd}
\usepackage{enumerate}

\usepackage{amssymb}
\usepackage{mathrsfs}
\usepackage{graphicx}

\newtheorem{thm}{Theorem}
\newtheorem{lem}{Lemma}
\newtheorem{prop}{Proposition}

\newtheorem{defn}{Definition}
\newtheorem{rem}{Remark}

\usepackage{enumitem}
\usepackage{mathtools,amssymb}
\usepackage{listliketab}
\newcounter{listliketablabel}
\newcommand*{\nextnum}{\addtocounter{listliketablabel}{1}(\thelistliketablabel)}
\makeatletter
\newcommand*{\storestyleofextra}[2][]{%
    \begin{lrbox}{\llt@list@box}
        \noindent
        \begin{minipage}{\linewidth}
            \begin{#2}[#1]
            \item[] \storeliststyle{}
            \end{#2}
        \end{minipage}
    \end{lrbox}\ignorespacesafterend
}
\makeatother

\storestyleofextra[noitemsep,label=\alph*)]{enumerate}

\begin{document}

\title{Historic  Behaviour for Random Expanding Maps on the Circle}

\author{Yushi Nakano}

%\affils{Osaka City University}

%\com{}

\date{}
%\subjclass{Primary 37C40; Secondary 37H10}

%\renewcommand{\keywordsname}{Key words}% when changing the name "Key words"
%\keywords{Historic behaviour; Random expanding maps}

%\authorname{Yushi Nakano}
%\address{Advanced mathematical institute, Osaka City University, Osaka,
%558-8585, JAPAN}
%\email{yushi.nakano@gmail.com}

\maketitle

\begin{abstract}
Takens constructed a residual subset of the state space consisting of initial points with historic  behaviour for expanding maps on the circle. We prove that this statistical property of expanding maps on the circle is preserved under small random perturbations. The proof is given by  establishing a random Markov partition, which follows from a random version of Shub's Theorem on topological conjugacy with the folding maps. 
%As a by-product, we obtain  a new formula for  the unique absolutely continuous ergodic invariant probability measure of random expanding maps on the circle.   
\end{abstract}

\maketitle

\section{Introduction}
Let $M$ be a compact smooth Riemannian manifold. For a dynamical system $f:M\to M$, the orbit  issued from an initial point $x$ in $M$ is said to have historic  behaviour  when there exists a continuous function $\varphi :M\to \R$ such that the time average
\[
\lim _{n\to \infty} \frac{1}{n} \sum ^{n-1} _{j=0} \varphi (f^j(x))
\]
does not exist. Several statistical quantities are given as the time average of some observable $\varphi$, 
and thus  it is  a natural question  in smooth dynamical systems theory whether the set of points  with historic  behaviour is not negligible in some sense;  it is typically formulated as a positive measure set with respect to the normalised Lebesgue measure. 
Takens  asked in \cite{Takens2008} whether there are persistent classes of smooth dynamical systems such that the set of initial points with historic  behaviour is of positive Lebesgue measure (\emph{Takens' Last Problem}).  
Very recently, it was affirmatively answered by Kiriki and Soma  \cite{KS2017} with diffeomorphisms having homoclinic tangencies.
The reader is asked to see   \cites{Ruelle2001, Takens2008, BDV04} for the background of historic  behaviour in this measure-theoretical sense; see also \cites{HK1990, LR2016}  for the (recent) contribution to  Takens' Last Problem for one-dimensional mappings and flows.

%The study of historic  behaviour is important but far from complete understanding. 

As another measurement to investigate historic  behaviour, 
Takens \cite{Takens2008} considered whether the set of points with historic  behaviour is a residual subset of the state space (i.e., the set of points with historic  behaviour is not negligible in a topological sense).
%, particularly whether it is a residual set.
For the doubling map on the circle, he constructed a residual subset consisting of initial points with historic  behaviour  by using symbolic dynamics. Since the method of symbolic dynamics is applicable to any expanding maps on the circle, the proof can be literally translated to expanding maps on the circle (see also \cite{BS2000} for another investigation of historic behaviour for expanding maps, and \cite{KLS2016} for historic behaviour of geometric Lorenz flows in topological sense).  
Our goal in this paper is to extend 
%the doubling map $E_2$ on the circle given by $E_2: x\mapsto 2x$ mod 1, and showed that the set of points for which $E_2$ has historic  behaviour is a residual set, i,e,.. 
%In this paper, we will extend 
his result for historic  behaviour (and its  generalisation to expanding maps on the circle) to a  random setting.
%: for any quenched random perturbation of an expanding map on the circle with small noise level, we can construct a residual set of initial points with historic  behaviour. The precise definition will be given in the next section.
For a random version of Takens' problem in measure-theoretical sense, we refer to the result by Ara\'ujo \cite{Araujo2000}; 
%Ara\'ujo  investigated some random perturbations of diffeomorphisms on a compact Riemannian manifold. I
\emph{in contrast to} our result in topological setting, he gave a negative answer to a random version of  Takens' problem in measure-theoretical sense for general diffeomorphisms under some conditions on noise.
%when noise in the systems is independent and identically distributed and satisfies several  conditions (called ``physical parametric noise" in \cite{Araujo}). 
%See \cite{Araujo} and \cite{Araujo2}.
%For a surprising (need nicer adj) (and contractive to our result (in some sense)) result for random version of historic  behaviour problem in measure-theoretical sense, we ask the reader to see  the result by Araujo \cite{Araujo}.

As in the  unperturbed case,
the key step in the proof is establishing a (random) Markov partition. This is given through proving a random version of Shub's Theorem on topological conjugacy of expanding maps with the folding maps. 
%As a by-product of this extension we will give a  formula for (the density function of) the unique absolutely continuous ergodic invariant probability measure of random expanding maps, which is a new formula  to the best of our knowledge. 

\subsection{Definitions and results}

Let $\Sone$ be a circle given by $\Sone =\R /\Z$. We endow $\Sone$ with a metric $d_{\Sone}(\cdot ,\cdot)$, where $d_{\Sone}(x, y)$ is the infimum of  $\vert \tilde x-\tilde  y \vert$ over all representatives $\tilde x,\tilde y$ of $x,y \in \Sone$, respectively.
Let ${\Ci}^{r}(\Sone ,\Sone )$ and $\mathrm{Homeo} (\Sone ,\Sone)$ be the spaces of all  endmorphisms of class $\mathscr C^r$ and homeomorphisms on the circle $\Sone$,  endowed with the usual  ${\Ci}^{r}$ and $\Ci ^0$ metrics $d_{\mathscr C^r} (\cdot ,\cdot)$ and $d _{\Ci ^0 }(\cdot ,\cdot )$, respectively, with $r >1$.
(Given that $r=k+\gamma$ for some $k\in \N$, $k\geq 1$ and $0\leq \gamma \leq 1$, $f\in \mathscr{C}^r(\Sone ,\Sone)$ denotes the $k$-th derivative of $f$ is $\gamma$-H\"older.) 
Let $\pi _{\Sone}:\R \to \Sone$ be the canonical projection on $\Sone$, i.e., $\pi _{\Sone} (\tilde x)$ is the equivalent class of $\tilde x\in \R$.  
We call $A$($\subsetneq \Sone$) a closed
interval of $\Sone$ if $A=\pi _{\Sone} (\tilde A)$ for some closed interval $\tilde A$ of $\R$. Open intervals or
left-closed and right-open intervals of $\Sone$ are defined in a similar manner.
We let $\mathcal F (\Sone)$ denote the space consisting of all nonempty left-closed and right-open intervals of $\Sone$,  all point sets  of $\Sone$ and the state space $\Sone$, endowed with the Hausdorff metric $d_H(\cdot ,\cdot )$. 
We endow ${\Ci}^{r}(\Sone ,\Sone )$, $\mathrm{Homeo} (\Sone ,\Sone)$ and $\mathcal F(\Sone)$ with the Borel $\sigma$-field.

Let $\Omega$ be a separable complete metric space endowed with the Borel $\sigma$-field $\mathcal B(\Omega)$ with a  probability measure $\mathbb P$. 
%\footnote{Probably completeness is not needed}
%We also assume that $\Omega$ is a standard probability space.  
Given a smooth map $f_0:\Sone \to \Sone $ of class $\mathscr{C} ^r$,
let $\{ f_{\epsilon} \} _{\epsilon >0}$ be a family of continuous mappings defined on $\Omega$ with values in $\mathscr{C} ^r(\Sone ,\Sone )$ such that
\begin{equation}\label{eq:assumption}
\sup\displaylimits_{\omega\in \Omega} d_{\mathscr{C}^r}(f_{\epsilon}(\omega) ,f_0) \rightarrow 0 \quad \mathrm{as} \;\epsilon \rightarrow 0.
\end{equation}
%(For a discussion on this condition, see the remark below  the proof Theorem \ref{prop:randomshub}.)
%\footnote{Can be weakened? See version 7 also.} 
For each $\epsilon >0$, 
adopting the notation $f_\epsilon (\omega ,\cdot)=f_\epsilon (\omega)$, 
the distance between $f_\epsilon(\omega ,x)$ and $f_\epsilon(\omega ^\prime ,x)$ is bounded by $d_{\mathscr{C}^r}(f_\epsilon (\omega),f_\epsilon (\omega ^\prime))$ for each $x\in \Sone$ and each $\omega , \omega ^\prime \in \Omega$. Thus,
it is straightforward to realise that $f_\epsilon :\Omega \times \Sone \to \Sone $ is a continuous (in particular, measurable) mapping.
When convenient, we will identify $f_0 :\Sone \to \Sone$ with the constant map
$ \Omega \ni \omega \mapsto f_0$. 

Let $f_0$  be an orientation-preserving expanding map on the circle, i.e., there exists a constant $\lambda _0>1$ such that
$
 \inf _{x}\frac{d}{dx} f_0 (x)    \geq \lambda  _0.
$
 For the properties of expanding maps, the reader is referred to \cite{KH95}. 
 Let $k\geq 2$ be the  degree of the covering map $f_0$.
In view of \eqref{eq:assumption} it then follows that $f_\epsilon (\omega)$ is an orientation-preserving  expanding map for each $\omega \in \Omega$ if $\epsilon $ is sufficiently small. In fact, if
$\lambda _0 = \inf _x  \frac{d}{dx} f_0 (x)$ and we set $\lambda =(\lambda _0 +1)/2$, then $\lambda >1$ and  we can find an $\epsilon _0 >0$ such that
\begin{equation}\label{assumption2}
\inf _{\omega} \inf _{x} \frac{\partial}{\partial x} f_\epsilon  (\omega , x)  \geq \lambda  , \quad  0\leq \epsilon <\epsilon _0 .
\end{equation}
Let $\delta _0<\frac{1}{2}(1-\frac{1}{k})$ be a positive number and $\eta $ a positive number such that 
\begin{equation}\label{eq:aaas}
\eta <\min\left\{ \frac{1}{2}, (\lambda -1)\delta _0 \right\}.
\end{equation}
We also assume  $\epsilon _0$ to be  sufficiently small such that
%\footnote{This condition SHOULD be relaxed although it is not very strange to be assumed.}
\begin{equation}\label{eq:knear}
\sup _{\omega \in \Omega} d _{\mathscr C^0} (f_\epsilon (\omega) ,f_0) < \eta, \quad 0\leq \epsilon <\epsilon _0.
\end{equation}
%It follows from  \eqref{eq:assumption} that if we take $\epsilon _0$ sufficiently small then
%Throughout the rest of this paper,  such that \eqref{assumption2} and \eqref{eq:knear} hold. 
%\footnote{Better to combine these two above and remove $\eta$?}
In particular, 
$f_\epsilon (\omega)$ is a covering map of $\Sone$ with degree $k$ for each $\omega \in \Omega$.

Since $f_0$ is a covering of $\Sone$ of degree $k\geq 2$, there exists a fixed point $p _0 \in \Sone$ for $f_0$. 
In view of \eqref{eq:assumption} and \eqref{assumption2} together with that  $f_0$ is locally a  diffeomorphism of the circle, we can find a closed interval $B=B_\epsilon$ including  $p_0$ such that if $0\leq \epsilon <\epsilon _0$, then $f_\epsilon (\omega ) :B \to f_\epsilon (\omega ) (B)$ is a diffeomorphism such that $B \subset f_\epsilon (\omega ) (B) $ for each $\omega \in \Omega$  (by taking $\epsilon _0$ sufficiently small if necessary). We assume that $\epsilon _0$ is sufficiently small such that 
 \begin{equation}\label{eq:knear2}
 \vert B \vert \leq \delta _0
 \end{equation} 
 for all $0\leq \epsilon <\epsilon _0$,
where $\vert B\vert$ is the  length of $B$ with respect to   $d_{\Sone}(\cdot ,\cdot)$. 
%From the condition \eqref{eq:knear2}, it follows that there exists a point which is in $\Sone$ but not in $B$. For simplicity we translate the point to   $0$.
%, so that $B$ is a closed interval strictly included in $[0,1)$. 
%(by virtue of \eqref{eq:knear2} with the fact that $\delta _0<\frac{1}{2}$) the distance $d_{\Sone}(x, y)$ coincides with $\vert  x - y\vert$ for each $x,y \in B$.
%, where $\bar x$, $\bar y$ are representatives in $[0,1)$ of $x,y$, respectively.

\begin{rem}
 When there is no ambiguity, the noise level $\epsilon$ will sometimes be omitted from the notation, in particular when the dependence on the noise parameter $\omega \in \Omega$ is already displayed. 
Throughout the rest of the paper we will also permit us to use $\epsilon _0$ as a way to denote the upper bound of a range $0\leq \epsilon <\epsilon _0$   for which \eqref{assumption2}, \eqref{eq:knear} and \eqref{eq:knear2} hold, even if $\epsilon _0$ may change between occurrences. %This will basically be  showcased only in   Theorem \ref{thm:BKS}.
\end{rem}

% By  \cite{Stafford1990}, there exists a Markov partition $\mathcal I_0 , \mathcal I_1, \ldots , \mathcal I_{k-1}$ of the expanding map $f_0$ on the circle  with some $k\geq 2$. Thus, it follows from \eqref{eq:assumption} and the continuity of $f_\epsilon $ that for each sufficiently small $\epsilon >0$, we can find a Markov partition $\mathcal I_0^{\epsilon ,\omega} ,\mathcal I _1^{\epsilon ,\omega },\ldots ,\mathcal I_{k-1}^{\epsilon ,\omega}$ of $f_\epsilon (\omega)$ such that the mapping $\omega \mapsto \overline{\mathcal I _j^\omega }$ from $\Omega$ to the space $\mathcal F(\Sone)$ of all closed and bounded subsets of $\Sone$, endowed with the Hausdorff metric,  is measurable  with respect to the Borel $\sigma$-field $\mathcal B(\mathcal F(\Sone))$ of $\mathcal F(\Sone)$ for each $0\leq j\leq k-1$. \footnote{WRITE the proof?}

Let $f: \Omega \to \Ci ^r(\Sone ,\Sone)$ be a continuous  mapping, and 
$\theta :\Omega \rightarrow \Omega$  a measure-preserving  homeomorphism  on $(\Omega ,\mathbb P)$ (see Remark \ref{rmk:3} for this condition). 
For simplicity, we also assume  $\theta$ to be ergodic. 
For each $n\geq 1$, let $f^{(n)} (\omega,x)$ be the fibre component in the $n$-th iteration of the skew product mapping
\begin{equation*}
\varTheta (\omega , x) =(\theta \omega ,f(\omega ,x)),\quad (\omega ,x) \in \Omega\times \Sone,
\end{equation*}
where we simply write $\theta \omega$ for $\theta (\omega)$. 
Setting the notation $f_\omega =f (\omega ,\cdot)$ and $f_{\omega}^{(n)}=f^{(n)}(\omega,\cdot)$, the explicit form of $f_{\omega}^{(n)}$ is
\begin{equation}\label{eq:fibre}
f^{(n)} _\omega =f _{\theta^{n-1}\omega } \circ f _{\theta^{n-2}\omega }\circ\cdots\circ f _\omega.
\end{equation}
For convenience, we set $f^{(0)}_\omega =\mathrm{id} _{\Sone}$ for each $\omega \in \Omega$.  
We call $\{ f^{(n)}_\omega (x)\} _{n\geq 0} = \{ x, f _\omega (x) ,f_{\omega} ^{(2)}(x) , \ldots \}$ the \emph{random orbit} of $f$ issued from $(\omega ,x) \in \Omega \times \Sone$. 
As in the treatment of the unperturbed case undertaken by Takens  \cite{Takens2008}, we now define historic  behaviour for random orbits. 
% \footnote{Extend this?}
\begin{defn}
We say that a random orbit issued from $(\omega ,x)$ has \emph{historic  behaviour} if  the empirical measure
\[
 \frac{1}{n} \sum _{j=0} ^{n-1} \delta _{f^{(j)} _ \omega (x)}
\]
does not converge to any probability measure on the circle in weak$^*$ sense, where $\delta _x$ is the Dirac measure at $x$.
\end{defn}
%\footnote{Should be defined for general measurable mapping $f:\Omega \to \Ci ^r(\Sone , \Sone)$?}

Our goal  is to prove the following theorem.
\begin{thm}\label{conj:1}
Let $0\leq \epsilon <\epsilon _0$. For $\mathbb P$-almost every $\omega \in \Omega$, we can find a residual subset $\mathcal R ^\omega$ of $\Sone$ such that for any $x\in \mathcal R^\omega$ the random orbit of $f_\epsilon$ issued from $(\omega ,x)$  has historic  behaviour. 
\end{thm}
%\begin{rem}
%Theorem \ref{conj:1} means that the set of  points whose random orbits have historic behaviour is large \emph{in the topological sense}. This is contractive to the result by \cite[Theorem A]{BKS96} that says the set is always of Lebesgue measure zero under the appearance of  the "physical" noise, even if the unperturbed dynamics has the set of points with historic  behaviour with positive Lebesgue measure. See  \cite{Araujo2000}, \cite{KirikiSoma} and \cite{Takens2008} for the detailed discussion for this topic. 
%\end{rem}
\begin{rem}
It is  natural to ask if it might be possible to replace ``$\mathbb P$-almost every'' with ``every'' in Theorem \ref{conj:1} since the assumption on our perturbations is given for every $\omega \in \Omega$ as in \eqref{eq:assumption}. 
However,  this strengthening of the conclusion in Theorem \ref{conj:1} seems to be impossible  if the proof utilises  the ergodicity of $(\theta ,\mathbb P)$ as in this paper, because the ergodicity of $(\theta , \mathbb P)$ (together with Birkhoff's ergodic theorem)  only ensures  that  time averages coincide with their corresponding space averages for $\mathbb P$-almost every $\omega \in \Omega$. 
%This  obstacle 
%for the improvement 
%only appears in a brief proof  of a theorem, s
See  Remark  \ref{rmk:3}   for the  detail. 
Another direction to improve Theorem \ref{conj:1} is to weaken the assumption on our perturbations,  
see  Remark \ref{rmk:4} for the detailed discussion about the direction.
\end{rem}

\section{The proof}
We start the proof by considering a random version of Dowker's Theorem. This theorem asserts that if we find a dense orbit with historic  behaviour (of the unperturbed dynamics), then there exists a residual subset  of the phase space such that the orbit of each point in the set has historic behaviour.   
For this purpose, we need to give stronger  forms of the definitions of denseness and historic  behaviour for random orbits.
% of a random variable with values in $\Sone$. 
When there is no confusion, we employ the notation $f_\omega^{(n)} =f ^{(n)}(\omega ,\cdot)$  for $f=f_\epsilon$ once  $0\leq \epsilon <\epsilon _0$ is given.
For a continuous function $\varphi : \Sone \to \R$, $\omega \in \Omega$, $x\in \Sone$ and $n\geq 1$, we define the truncated time average $B_n (\varphi ; \omega ,x)$ of the observable $\varphi$ by
\[
B_n(\varphi ; \omega ,x)=\frac{1}{n} \sum _{j=0} ^{n-1} \varphi \circ f_\omega ^{(j)} (x). 
\]
Note that the orbit issued from $(\omega ,x)$ has historic  behaviour if $B_n(\varphi ;\omega ,x)$ does not converge for some continuous function $\varphi :\Sone \to \R$.

\begin{defn}
Let $f:\Omega \to \Ci ^r (\Sone ,\Sone)$ be a measurable mapping. Let   $X$ be a random variable on  $\Omega$ with values in $\Sone$. We call  
\[
\{ X(\omega) ,f_{\omega}(X(\omega)) , f_{\omega} ^{(2)} (X(\omega)) ,\ldots \}
\]
 the \emph{future random orbit} of $X$ at $\omega$, and 
\[
\{ X(\omega) ,f_{\theta ^{-1} \omega}(X(\theta ^{-1} \omega)) , f_{\theta ^{-2}\omega} ^{(2)} (X(\theta ^{-2} \omega)) ,\ldots \}
\]
 the \emph{past random orbit}  of $X$ at $\omega$ (see \cite{BBM-D} for this notation).

We say that $X$ has \emph{historic  behaviour} at $\omega$ if there exist real numbers $\alpha , \beta$ and a continuous function $\varphi :\Sone \to \R$ such that 
\begin{equation}\label{eq:strongHB}
\liminf _{n\to \infty} B_n (\varphi ; \theta ^{-\ell}\omega ,X(\theta ^{-\ell}\omega)) <\alpha <\beta <\limsup _{n\to \infty} B_n(\varphi ; \theta ^{-\ell} \omega ,X(\theta ^{-\ell} \omega))
\end{equation}
for  each $\ell \geq 0$ (that is, the time average of $\varphi$ along the future random orbit of $X$  does not exist at each $\theta ^{-\ell}\omega$).
\end{defn}

%For $\omega \in \Omega$,  let $\mathcal R ^\omega$ be the set of all points $x\in \Sone $ such that the random orbit issued from $(\omega ,x) $  has historic  behaviour.  
%\footnote{MEASURABILITY of $\mathcal R^\omega$. Probably the compliment of $\mathcal R^\omega$ can be easily shown to be measurable by reiterating the argument for the measurability of $\mathcal I^\omega $, and so is $\mathcal R^\omega$. Or now (changed the statement of the prop 4), we even do not need the measurability?}

\begin{prop}\label{randomDowker}
Let $f:\Omega \to \Ci ^r (\Sone ,\Sone)$ be a measurable mapping and $\omega \in \Omega$. Assume that there exists a measurable mapping $X:\Omega \to \Sone$  such that the past random orbit of  $X$  at $\omega$ is  dense in $\Sone$  and $X$ has historic  behaviour  at $\omega$. 
Then, we can find a residual subset $\mathcal R ^\omega$ of $\Sone$ such that for any $x\in \mathcal R^\omega$,  the random orbit  issued from $(\omega ,x)$  has historic  behaviour. 
\end{prop}
\begin{proof}
Let $\alpha ,\beta$ and $\varphi$ be constants and a continuous function given in \eqref{eq:strongHB} for a measurable mapping $X:\Omega \to \Sone$ and $\omega \in \Omega$, where $X$ has historic  behaviour at $\omega$ by hypothesis. We let $\mathcal R^\omega =s_1^\omega \cap s_2^\omega$, where 
\begin{align}\label{eq:dda3}
&s_1^\omega = \bigcap _{N\geq 1} \bigcup _{n\geq N} \left\{ x\in \Sone : B_{n} (\varphi ;  \omega ,x) <\alpha  \right\},\\ 
&s_2^\omega = \bigcap _{N\geq 1} \bigcup _{n\geq N} \left\{ x\in \Sone : B_{n} (\varphi ; \omega ,x) >\beta\right\}  .
\end{align}
%For simplicity we write $s_1^\omega$ and $s_2^\omega$ for the set in the left big parenthesis and the set in the right big parenthesis in \eqref{eq:dda3}, respectively, i.e., $\mathcal R^\omega =s^\omega _1 \cap s^\omega _2$.  
Then it is straightforward to see that for any $x\in \mathcal R^\omega$, we have
\begin{equation*}
\liminf _{n\to \infty} B_n (\varphi ; \omega ,x) \leq \alpha <\beta \leq \limsup _{n\to \infty} B_n(\varphi ; \omega , x ).
\end{equation*}
In particular, the random orbit  issued from $(\omega ,x)$  has historic  behaviour. 
%Then 
%Let $\Omega _0$ be the set of $\omega$'s such that \eqref{eq:strongHB} holds. Then 
%it is straightforward to see that   $X(\omega)$ is $\mathbb P$-almost surely in $\mathcal R^\omega$.
%(Indeed.. or more precisely..)
%Conversely, 
%for each $\omega \in \Omega$.
%Consequently, for any $\omega \in \Omega$ and $x\in \mathcal R ^\omega$, the random orbit issued from $(\omega ,x)$ has historic  behaviour. 
Note also  that
$\mathcal R ^\omega$ is a countable intersection of open sets.

The rest of the proof is devoted to 
 showing that   the past random orbit  of $X$ at $\omega$  is included in  $\mathcal R ^{ \omega}$. 
Since the past random orbit of $X$ at $\omega$ is  dense in $\Sone$ by hypothesis,  
it leads to  that $\mathcal R^\omega$ is dense in $\Sone$, and we complete the proof.
Let $\ell \geq 1$. 
Due to the observation $f_\omega ^{(i)} \circ f_{\theta ^{-\ell} \omega} ^{(\ell)} =f^{(i+\ell )} _{\theta ^{-\ell }\omega} $ for any $i\geq 0$, we have
\begin{equation*}
B_n\left(\varphi  ;\omega , f^{(\ell )} _{\theta ^{-\ell }\omega} (X(\theta ^{-\ell} \omega))\right) 
=\frac{1}{n} \sum _{j=\ell} ^{n+\ell-1} \varphi \circ f^{(j)} _{\theta ^{-\ell }\omega} (X(\theta ^{-\ell} \omega)).
\end{equation*}
Therefore, for each $n\geq \ell$ 
%&=\frac{1}{n} \sum _{j=0} ^{n-1} \varphi \circ f_{\omega} ^{(j+i)} (x_0) \\
\begin{multline*}
B_n\left(\varphi  ;\omega , f^{(\ell )} _{\theta ^{-\ell }\omega} (X(\theta ^{-\ell} \omega))\right) -B_n(\varphi ;\theta ^{-\ell} \omega ,X(\theta ^{-\ell} \omega)) \\
= \frac{1}{n} \left( \sum _{j=n} ^{n+\ell-1} \varphi \circ f_{\theta ^{-\ell }\omega} ^{(j)}(X(\theta ^{-\ell} \omega))  - \sum _{j=0} ^{\ell -1} \varphi \circ f_{\theta ^{-\ell }\omega} ^{(j)}(X(\theta ^{-\ell} \omega))\right).
\end{multline*}
Hence we have
\begin{equation}\label{eq:dda1}
\left\vert B_n\left(\varphi  ;\omega , f^{(\ell )} _{\theta ^{-\ell }\omega} (X(\theta ^{-\ell} \omega ))\right) -B_n(\varphi ;\theta ^{-\ell} \omega ,X(\theta ^{-\ell} \omega)) \right\vert \leq \frac{2\ell \Vert \varphi \Vert _{\Ci ^0}}{n},
\end{equation}
which goes to zero as $n$ goes to infinity for any fixed $\ell \geq 1$.

Let
   $\{ n_k\} _{k\geq 1} $ be a sequence such that
%$$ for any real number $\rho$ between  
$B_{n_k}(\varphi ;\theta ^{-\ell} \omega ,X(\theta ^{-\ell} \omega)) $
%, be a subsequence such that  
 converges to 
the infimum limit in \eqref{eq:strongHB}
%of   $\{B_n (\varphi ; \theta ^{-\ell} \omega ,X(\theta ^{-\ell}\omega)) \}_{n\geq 1}$
 as $k$ goes to infinity. 
Then, in view of 
%\eqref{eq:strongHB} with $\omega$ replaced by $\theta ^{-\ell} \omega$, together with 
\eqref{eq:dda1} we have
\[
B_{n_k}\left(\varphi  ;\omega , f^{(\ell )} _{\theta ^{-\ell }\omega} (X(\theta ^{-\ell} \omega ))\right) <\alpha \quad \text{if $k$ is sufficiently large.}
\]
This implies that $f^{(\ell)} _{\theta ^{-\ell}  \omega} (X(\theta ^{-\ell} \omega)) $ is in $s_1^\omega$. 
% is in the set in the left big parenthesis of \eqref{eq:dda3}.
In a similar manner, we can show that $f^{(\ell)} _{\theta ^{-\ell}  \omega} (X(\theta ^{-\ell} \omega))$ is in $s_2^\omega$. Therefore, $f^{(\ell)} _{\theta ^{-\ell}  \omega} (X(\theta ^{-\ell} \omega))$  is in $\mathcal R^\omega$, and  the past random orbit of $X$ at $\omega$ is included in $\mathcal R^\omega$.
%\footnote{Comment on full measure?}
%Finally, it follows from the denseness of the random orbit  of $X$ that $\mathcal R^\omega$  is $\mathbb P$-almost surely a dense subset of  $\Sone$,
%%. 
%%Since $\mathcal R^\omega$ is a countable intersection of open sets such that for any $x\in \mathcal R ^\omega$, the random orbit issued from $(\omega ,x)$ has historic  behaviour, 
%and we complete the proof.
%\footnote{improve English more?}
% $\mathcal R^\omega$ is $\mathbb P$-almost surely a residual subset of $\Sone$. 
%, and so is $\mathcal R^\omega$.
\end{proof}
%\begin{rem}
%The pointwise version  of Proposition \ref{randomDowker}  might \emph{not} be correct. That is, we do not know whether the following is true: "Given $\omega \in \Omega$, assume that there exists a   point $x_0\in \mathcal R^\omega$ such that the  random orbit  issued from $(\omega ,x _0)$  is dense in $\Sone$. Then, $\mathcal R^\omega$ is a residual subset of $\Sone$."
%%it migh if we only assume the condition that "for $\omega$, we can find a point $x_0=x_0(\omega)$ such that the random orbit issued from $(\omega ,x_0)$ has historic  behaviour and is dense in $\Sone$", then we may \emph{not} get the conclusion of \ref{randomDowker}. 
%\end{rem}

Due to Proposition \ref{randomDowker}, Theorem \ref{conj:1} is reduced to constructing \emph{one} measurable mapping $X$ such that  for $\mathbb P$-almost every $\omega \in \Omega$, the past random orbit of $X$ at $\omega$ is  dense in $\Sone$ and $X$ has historic  behaviour at $\omega$.

\subsection{Shub's Theorem and Markov partition}

In the next subsection, we establish the coding of \emph{graphs} 
%including a given point $(\omega ,x)$ 
associated with  a \emph{random Markov partition} of $f_\epsilon$, which is a key ingredient in our proof. 
%Random Markov partition will be constructed by a random version of Shub's theorem below
For that purpose,
 %of constructing a Markov partition of $f_\epsilon$, 
 we will need the following  extension of Shub's Theorem  on topological conjugacy between expanding mappings and the folding mapping with the same degree
 % (see \cite{Shub1970} for detail)
  to our random setting. 
  
Set $\mathcal A =\{ 0,1 ,\ldots ,k-1\}$.  
For  $s=(s_0,s_1, \ldots , s_m,s_{m+1},\ldots , s_{n},\ldots ,s_{N-1})\in \mathcal A^{N}$ with integers $0\leq m\leq n\leq N-1$, 
we  denote   $(s_m,s_{m+1},\ldots , s_{n})$   by $[s]_m^{n}$. 
Recall that in view of \eqref{eq:knear}, the degree of $f_\epsilon (\omega ,\cdot):\Sone \to\Sone$ is $k$ if $0\leq \epsilon <\epsilon _0$ and $\omega \in \Omega$.
\begin{thm}\label{prop:randomshub}
Suppose that $0\leq \epsilon <\epsilon _0$. 
%\footnote{OK?}
Then we can find a continuous mapping $h :\Omega \to \mathrm{Homeo}(\Sone ,\Sone)$ (in particular, measurable mapping)  which satisfies that 
\begin{equation}\label{eq:conjugacy}
h (\theta \omega) \circ E_k = f_\epsilon (\omega ) \circ h(\omega) 
\end{equation}
for any $\omega \in \Omega$, where $E_k :\Sone \to \Sone $ is the $k$-folding mapping defined by $E_k(x) =\pi _{\Sone}(k\tilde x)$ for each $x\in \Sone$ with a representative $\tilde x$ of $x$.

Furthermore, for every $\omega ,\omega ^\prime \in \Omega$ we have
\begin{equation}\label{eq:knear3}
d_{\Ci ^0 } (h(\omega) ,h(\omega ^\prime)) \leq  \delta _0,
\end{equation}
where $\delta _0$ is given in \eqref{eq:aaas}.
\end{thm}
%\begin{rem}
%For the purpose of proving Theorem \ref{conj:1} only, a weaker version of Theorem \ref{prop:randomshub} might be enough. \footnote{Check.} However, (full statement of) Theorem \ref{prop:randomshub} will be helpful for the reader to understand the dynamics of randomly perturbed expanding maps on the circle, and thus this was carried here. 
%\end{rem}
\begin{proof}
 %As in the construction of topological conjugacy for the unperturbed expanding maps established by Shub \cite{Shub1970}, we employ a lift $\tilde f_\epsilon (\omega) :\mathbb R \to \mathbb R$ of $f_ \epsilon (\omega)$ for fixed $\omega \in \Omega$. The difference with the unperturbed case is that the existence of the fixed point of $f_\epsilon$ is now not trivial.
 %which preserves an invariant continuous mapping of $f_\epsilon$. 
 Throughout the proof, we fix $0\leq \epsilon <\epsilon _0$ and omit it  in  notions if there is no confusion. Recall that $B$ is a  closed interval of $\Sone$ (given above \eqref{eq:knear2}) such that  $B\subset f_\epsilon(\omega)(B)$ for each $\omega\in\Omega$.
Then, one can find a lift $\tilde f_\epsilon (\omega):\R\to \R$ of the $k$-covering $f_\epsilon (\omega):\Sone \to \Sone$ 
%(so that $\tilde f_\epsilon (\omega)(\tilde x+1)=\tilde f_\epsilon (\tilde x)+k$ for each $\tilde x\in \R$)
 with $\omega \in \Omega$  and a closed  interval $\tilde B\subset \R$ such that $\pi _{\Sone}(\tilde B)=B$ and $\tilde B\subset \tilde f_\epsilon (\omega)(\tilde B)$ for each $\omega \in \Omega$.
  
 We shall first  find a  continuous mapping $p : \Omega \to \tilde B$, which is invariant under  $\tilde f_\epsilon$, i.e., $\tilde f_\epsilon (\omega , p (\omega)) = p (\theta \omega)$ for every $\omega \in \Omega$. 
 %and $f_\epsilon (\omega) :B\to f_\epsilon(\omega)(B)$ is a diffeomorphism for every $\omega \in \Omega$.  
 We denote for each $\omega \in \Omega$ the inverse mapping of  the diffeomorphism $\tilde f_\epsilon (\omega)$ on $\R$  by $F(\omega)$.
 %, i.e., $F(\omega ) :f_\epsilon (\omega) (B) \to B$. 
 Then, 
% % $F(\omega ):B\to B$ is also well defined since $B\subset f_\epsilon (\omega) (B)$ for every $\omega \in \Omega$.
% %Hence,
%  with the notation of the space of all $\Ci ^r$ diffeomorphisms on $\R$  by $ \mathrm{Diff}^r (\R, \R)$, it follows from the continuity of $f_\epsilon$ that $F:\Omega \to \mathrm{Diff} ^r (\R,\R)$ is continuous.
% % coincides with the inverse branch of $f_\epsilon (\omega)$ restricted on $B$. 
%% \footnote{More comment?} 
% Moreover,
  due to \eqref{assumption2}, 
   %for each $0\leq \epsilon <\epsilon _0$, 
   we have
\begin{equation}\label{eq:branchcontraction}
 \sup _{\omega} \sup _x \left\vert  \frac{d}{dx}\left[ F(\omega) \right](x) \right\vert \leq \lambda ^{-1}. 
\end{equation}
Let $\mathscr C^0(\Omega , \tilde B)$ be the space of all continuous mappings from $\Omega $ to the closed interval $\tilde B\subset \R$, endowed with the usual $\mathscr C^0$ metric $d_{\mathscr C^0} ( \cdot ,\cdot ) $  defined by $d_{\mathscr C^0}(\phi _1 , \phi _2) = \sup _{\omega} \vert \phi _1(\omega ) -\phi _2(\omega) \vert $ for $\phi _1, \phi _2 \in \mathscr C^0(\Omega , \tilde B)$. 
%This is a complete metric since $D$ is a closed interval which does not include  $0$ with diameter $\leq \delta _0<1/2$. \footnote{CITATION? or just delete}
%, and $\mathscr M(\Omega , D)$ the  completion of $\mathscr C(\Omega , D)$  with respect to $\Vert \varphi \Vert$. \footnote{If $p _\epsilon$ can be continuous, somehow redundant argument?} 
For each $\varphi \in \Ci ^0 (\Omega ,\tilde B)$, we define a mapping $\mathcal G (\varphi) :\Omega \to \tilde B$ by
\[
 \mathcal G (\varphi )  (\omega) = F( \omega) \left( \varphi (\theta  \omega)\right) , \quad \omega \in \Omega.
\]
(Note that $ \mathcal G (\varphi )  (\omega) $ is indeed in $\tilde B$ since  $\tilde B\subset \tilde f_\epsilon (\omega)(\tilde B)$ for each $\omega \in \Omega$.)
It is straightforward to see that $ (\omega ,x) \mapsto F(\omega )(x) $ is a continuous mapping from $\Omega \times \tilde B$ to $\tilde B$ (see the argument below \eqref{eq:assumption}).
Thus,  it follows from the continuity of $\theta$ and $\varphi$ that
 $\mathcal G (\varphi):\Omega \to \tilde B$ is also a continuous mapping.
 (The transformation $\mathcal G :\Ci ^0 (\Omega ,\tilde B) \to \Ci ^0 (\Omega ,\tilde B)$ is called the \emph{graph transformation} induced by $F$.)
  %\footnote{CITATION}) 
Furthermore, by virtue of \eqref{eq:branchcontraction} together with the mean value theorem, we have 
\begin{align*}
d_{\Ci ^0} (\mathcal G (\phi _1) ,\mathcal G (\phi _2) )
&\leq \sup _{\omega\in\Omega}  \left\vert F(\omega) (\phi _1(\theta \omega)) -F(\omega) (\phi _2(\theta \omega))\right\vert \\
&\leq \lambda ^{-1} \sup _{\omega\in\Omega} \left\vert \phi _1(\theta \omega) - \phi _2(\theta \omega)\right\vert =
\lambda ^{-1} d_{\Ci ^0} ( \phi _1 ,\phi _2), 
\end{align*}
for all $ \phi _1,\phi _2 \in \Ci ^0 (\Omega ,\tilde B)$, 
i.e.,  $\mathcal G$ is a contraction mapping on the complete metric space $\Ci ^0(\Omega ,\tilde B)$.  Therefore,  there exists a unique fixed point $p$ of $\mathcal G$. By construction, $p:\Omega \to \tilde B \subset \R$ is an $\tilde f_\epsilon$-invariant continuous mapping.

We next construct a sequence $\{  h_n \} _{n\geq 0}$ of  continuous mappings $h_n :\Omega \to \mathrm{Homeo}(\Sone , \Sone)$, which shall uniformly converge to  the desired mapping $h$ as $n$ goes to infinity. 
%We construct  $h_n (\omega) $ as the projection of a continuous, piecewise linear function $\tilde h_n :\Sone  \to \R$ (i.e., $h_n(\omega) = \pi \circ \tilde  h_n$ where $\pi :\R \to \Sone$ is the natural projection) induced by a correspondence between finite sets $\{ a^n _j (\omega) \} \subset \Sone$, $n\geq 0$, $$, 
%, and $\tilde h_n$ will be contruction  is defined 
Since $\tilde f _\epsilon (\omega) :\R \to \R$  for $\omega \in \Omega$ is a lift of the orientation-preserving $k$-covering $f_\epsilon (\omega)$ of $\Sone$ and $\tilde f_\epsilon (\omega )( p (\omega)) =p (\theta \omega)$, we have $\tilde f_\epsilon (\omega )(p (\omega)+1) =p(\theta \omega) +k$ and  $\tilde f_\epsilon (\omega) :[p(\omega) ,p(\omega) +1] \to [p(\theta \omega) ,p(\theta \omega) +k]$ is a monotonically increasing homeomorphism for all $\omega$. Now we  define points $ a _j^n (\omega) $ in $[p(\omega),p(\omega) +1)$ for $n\geq 0, 0\leq j\leq k^n-1$ and $\omega \in \Omega$, inductively  with respect to $n$. 
In the case $n=0$, let $a_0^0(\omega) =p(\omega)$ for each $\omega \in \Omega$. For given integer $n\geq 0$, we  assume that  $a^n _j(\omega)$'s are well defined  for each $\omega \in \Omega$ and $0\leq j\leq k^n-1$. Then we define $a^{n+1} _j(\omega)$ for $0\leq j\leq k^{n+1} -1$ of the form  $j=\ell  \cdot k^n +j_1$ with integers $0\leq \ell \leq k-1$ and $0\leq j _1\leq k^{n} -1$ by
\begin{equation}\label{eq:defofa}
a_j ^{n+1} (\omega)=F(\omega)(a^n _{j_1 } (\theta \omega) +\ell ),\quad \omega \in \Omega.
\end{equation}
Setting the notation $\bar s=\sum _{i=0}^{n-1}s _i \cdot k^i$ for $s =(s _0,s _1,\ldots ,s _{n-1}) \in \mathcal A^n$ and $F_\ell (\omega)=F(\omega)(\cdot + \ell)$,
the explicit form of $a^n_{\bar s}(\omega)$ with $s\in \mathcal A^n$ and $\omega \in \Omega$ is
\begin{equation}\label{eq:revisa1}
a_{\bar s}^n(\omega) =F_{s _{n-1}} (\omega)\circ F_{s _{n-2}}(\theta \omega) \circ \cdots \circ F_{s_0}(\theta ^{n-1}\omega)(p(\theta ^n\omega)).
\end{equation}
%(Recall that $a_0^0(\theta ^n\omega) =p(\theta ^n\omega)$.) 
We denote the composition of mappings in the right-hand side of \eqref{eq:revisa1} by $F^{(n)}_s (\omega)$, i.e., $a_{\bar s}^n(\omega) = F^{(n)}_s (\omega)(p(\theta ^n\omega))$.  
Then it is straightforward to see that 
$F_s^{(n+m)}(\omega) =F_{[s]_m^{n+m-1}}^{(n)}(\omega) \circ F_{[s]_0^{m-1}}^{(m)}(\theta ^n\omega)$ for each integer $m\geq 0$, $s\in \mathcal A^{n+m}$ and $\omega \in \Omega$.
Therefore,  
for any $s\in \mathcal A^n$,  integer $m\geq 0$ and  $\omega \in \Omega$, 
\begin{align}\label{eq:revisa2}
a^{n+m}_{k^m \cdot \bar s} (\omega) &=F_{s}^{(n)}(\omega) \circ F_{(00\ldots 0)}^{(m)}(\theta ^n\omega) (p(\theta ^{n+m}\omega))\\
&= F_{s}^{(n)}(\omega) (p(\theta ^n\omega))= a^n _{\bar s} (\omega).
\end{align}
(Note that $k^m\cdot \bar s = \bar t$ with the concatenation $t$ of two words $(00\ldots 0) \in \mathcal A^m$ and $s$,  and that $F_0(\omega)(p(\theta \omega)) =p(\omega)$ for any $\omega \in \Omega$.)
For convenience, we set $a^n_{k^n}(\omega) =p(\omega) +1$ for each $\omega \in \Omega$. 

For the time being, we fix $n\geq 0$ and $\omega \in \Omega$. Then it follows from the monotonicity of $\tilde f_\epsilon (\omega)$ in \eqref{eq:defofa} that $a_0^n(\omega)<a_1^n(\omega) <\cdots <a_{k^n}^n(\omega)$. 
%the sequence $\{ a_j^n(\omega) \} _{0\leq j\leq k^n}$ is monotonically increasing for each $n\geq 0$ and $\omega \in \Omega$. 
Thus  we can define a continuous, monotonically increasing, piecewise linear mapping $\tilde h_n (\omega) :[0,1] \to [p(\omega) ,p(\omega) +1]$ such that
\begin{equation}\label{eq:defofa2}
\tilde h_n(\omega) \left( \frac{j}{k^n} \right) = a_j^n(\omega), \quad 0\leq j\leq k^n ,
\end{equation}
and  that $\tilde h_n(\omega) :[j/k^n ,(j+1)/k^n] \to [a_j^n (\omega) ,a_{j+1} ^n (\omega)]$ is an affine  mapping of slope $k^n \cdot (a^n _{j+1} (\omega) -a_j^n(\omega))$.
%,  then $\tilde h_n(\omega):[0,1] \to [p(\omega) ,p(\omega) +1]$
%is monotonically increasing for each $n\geq 0$ and $\omega \in \Omega$. 
We finally define $h_n(\omega ): \Sone \to \Sone$  by 
\begin{equation}\label{eq:defofa3}
h_n (\omega)(x) =\pi _{ \Sone} \circ \tilde h_n (\omega)(\tilde x) , \quad x\in \Sone,
\end{equation}
where $\tilde x$ is  a representative of $x$ in $[0,1]$. (This is well-defined since $\tilde h_n(\omega)(1) =\tilde h_n(\omega)(0) +1$.)
%the natural projection from $\R$ to $\Sone$. 

Now we change $n\geq 0$   while $\omega $ is still fixed. 
For any nonnegative integers $n$, $m$ and $s\in \mathcal A^n$, it follows from \eqref{eq:revisa2} and \eqref{eq:defofa2} that
\[
\tilde h_{n+m}(\omega) \left(\frac{\bar s}{k^n}\right) = \tilde h_{n+m} (\omega)  \left(\frac{k^m\cdot \bar s}{k^{n+m}}\right) =a_{k^m \cdot \bar s}^{n+m}(\omega) =a_{\bar s}^n(\omega) = \tilde h_n (\omega) \left(\frac{\bar s}{k^n}\right) .
\]
Hence, by the monotonicity of $\tilde h_n(\omega)$ and $\tilde h_{n+m}(\omega)$, the supremum norm of $ h _n(\omega)(x) -h_{n+m}(\omega)(x)$ (over $x\in \Sone$) is bounded by 
\begin{equation}\label{eq:revisa3}
\max _{0\leq j\leq k^n-1} \left\vert a^n_{j+1} (\omega) -a^n_j(\omega) \right\vert .
\end{equation}
On the other hand,  since $[a_j^n(\omega) ,a _{j+1} ^n(\omega)]$ is diffeomorphically mapped onto $[p(\theta ^n\omega) +j,p(\theta ^n\omega)+j+1]$ by $\tilde f^{(n)}_\omega$,   the length of each $[a_j^n(\omega) ,a _{j+1} ^n(\omega)]$ does not exceed $\lambda ^{-n}$ (independently of $\omega$) by  virtue of \eqref{assumption2}, so that 
$\{\pi_{\Sone}( a^n _j (\omega) )\mid n\geq 0 , 0\leq j\leq k^n \}$ is dense in $\Sone$ and \eqref{eq:revisa3} uniformly converges to zero as $n$ goes to infinity with respect to $\omega$.
 Therefore,  $\{h_n \} _{n\geq 0}$ is a Cauchy sequence of continuous mappings from $\Omega$ to $\mathrm{Homeo} (\Sone ,\Sone )$ and there exists  the limit mapping $ h=\lim _{n\to \infty} h_n$, which is also by construction  a continuous mapping from $\Omega$ to $\mathrm{Homeo} (\Sone ,\Sone )$.
 %homeomorphism  on the circle. 
 %Moreover,
%
%We next check the measurability of $h:\Omega \to \mathrm{Homeo}(\Sone ,\Sone)$. 
%Reiterating the argument below \eqref{eq:assumption} on the continuity of $\Omega \times \Sone \ni (\omega ,x) \mapsto f_\epsilon (\omega ,x)$, it is deduced that the function
%$\Omega \times \R \ni (\omega ,x) \mapsto (\tilde f_\epsilon (\omega) )^{-1} (x)$
%is continuous. 
%Thus, by the virtue of 
%the continuity of $p:\Omega \to \Sone$ and $\theta :\Omega \to \Omega$  together with \eqref{eq:defofa},   $a^n _j :\Omega \to \R$ is a continuous function for each $n\geq 0$ and $0\leq j\leq k^n $. This gives  the continuity (measurability, in particular) of $h_n:\Omega \to \mathrm{Homeo} (\Sone ,\Sone)$ for each $n\geq 0$, and finally we get the measurability of $h :\Omega \to \mathrm{Homeo}(\Sone ,\Sone)$.

We see the conjugacy \eqref{eq:conjugacy} between $f_\epsilon $ and $E_k$. 
Let $\omega \in \Omega$, $n\geq 0$ and $0\leq j\leq k^{n+1} -1$ of the form  $j=\ell  \cdot k^n +j_1$ with some integers $0\leq \ell \leq k-1$ and $0\leq j _1\leq k^{n} -1$. Then on the one hand, 
in view of \eqref{eq:defofa}, \eqref{eq:defofa2} and \eqref{eq:defofa3}, we have
\[
f_\epsilon (\omega) \circ h (\omega) \circ \pi _{\Sone}\left( \frac{j}{k^{n+1}} \right) = \pi _{\Sone} \circ \tilde f_\epsilon (\omega) (a_{j}^{n+1} (\omega)) = \pi _{\Sone} \left(a_{j_1} ^n(\theta \omega)\right).
\]
On the other hand, 
\[
h(\theta \omega) \circ E_k \circ \pi _{\Sone}\left( \frac{j}{k^{n+1}} \right) = h(\theta \omega )\circ \pi _{\Sone} \left( \frac{j _1}{k^{n}} \right) = \pi _{\Sone} \left( a_{j_1} ^n(\theta \omega) \right).
\]
Since $\{ \pi _{\Sone}(j /k^{n}) \mid n\geq 0, 0\leq j\leq k^n -1\} $ is  dense in $\Sone$, we get \eqref{eq:conjugacy}.

In the end, we prove  the estimate  \eqref{eq:knear3} of $h$ by showing  
\begin{equation}\label{eq:knear4}
\vert a _j^n (\omega ) - a_j^n (\omega ^\prime)  \vert \leq \delta _0 
\end{equation}
for all $n\geq 0$, $0\leq j\leq k^n -1$ and $\omega , \omega ^\prime \in \Omega$.
We   show \eqref{eq:knear4}  by induction with respect to $n \geq 0$.
Due to \eqref{eq:knear2}, this inequality holds for $n=0$. Suppose that \eqref{eq:knear4} is true for given $n\geq 0$. Let $\omega, \omega ^\prime \in \Omega$ and $0\leq j\leq k^{n+1} -1$ of the form $j=\ell  \cdot k^n +j_1$ with some integers $0\leq \ell \leq k-1$ and $0\leq j _1\leq k^{n} -1$.  
Then, it follows from \eqref{eq:defofa} together with the triangle inequality that
$
\vert a _j^{n+1} (\omega ) - a_j^{n+1} (\omega ^\prime)  \vert 
$ is bounded by
$S_1 +S_2$, 
where
\begin{multline*}
S_1 =\left\vert F_\ell(\omega)(a^n _{j_1 } (\theta \omega)  ) -F_\ell(\omega^\prime)(a^n _{j_1 } (\theta \omega)  ) \right\vert ,\\
S_2= \left\vert F_\ell(\omega^\prime)(a^n _{j_1 } (\theta \omega) ) -F_\ell(\omega^\prime)(a^n _{j_1 } (\theta \omega ^\prime)  )\right\vert .
\end{multline*}
To estimate $S_1$, we let
$
x=F_\ell(\omega)(a^n _{j_1 } (\theta \omega) )$ and $x^\prime =F_\ell(\omega^\prime)(a^n _{j_1 } (\theta \omega)  ) .
$
%for the first and second term in the absolute value of  $s_1$, respectively. 
Then, $\tilde f_\epsilon (\omega ^\prime) (x^\prime) = \tilde f_\epsilon (\omega ) (x) $, and in view of  \eqref{eq:knear} we have
\[
\left\vert \tilde f_\epsilon (\omega ^\prime) (x) - \tilde f_\epsilon (\omega ^\prime) (x^\prime) \right\vert = \left\vert \tilde f_\epsilon (\omega ^\prime) (x) - \tilde f_\epsilon (\omega ) (x) \right\vert  \leq \eta.
\]
Hence, it follows from \eqref{assumption2} and the mean value theorem that 
\[
S_1=\vert x-x^\prime \vert \leq \lambda ^{-1} \left\vert \tilde f_\epsilon (\omega ^\prime) (x) - \tilde f_\epsilon (\omega ^\prime) (x^\prime) \right\vert  \leq \lambda ^{-1} \eta .
\]
On the other hand, by  virtue of \eqref{eq:branchcontraction} and the mean value theorem together with the hypothesis of induction, 
%\footnote{Better to write "for $n=m$"..?} 
we get
\[
S_2 \leq  \lambda ^{-1} \left\vert (a^n _{j_1 } (\theta \omega) +\ell ) -(a^n _{j_1 } (\theta \omega ^\prime) +\ell )\right\vert \leq \lambda ^{-1} \delta _0.
\]
These estimates together with the condition \eqref{eq:aaas} on $\eta$ and $\delta _0$ leads to that
$
\vert a _j^{n+1} (\omega ) - a_j^{n+1} (\omega ^\prime)  \vert  \leq 
\delta _0$, which completes the proof of the claim \eqref{eq:knear4}.
%\footnote{Need 2 of $\eta /2$?}
\end{proof}

%
%
%\begin{rem}
%It is 
%In a similar manner to the argument in the proof above, we can see that $\omega \mapsto (h(\omega))^{-1} (A)$ is measurable mapping from $\Omega$ to $\mathcal F(\Sone)$ for all $A\in \mathcal F(\Sone)$. Indeed, it is straightforward to get the claim with $h$ replaced by  $h_n:\Omega \to \mathrm{Homeo} (\Sone ,\Sone)$ for any $n\geq 0$, so the claim immediately follows. Therefore, if we denote the pushforward of given measure $\nu$ on $\Sone$ under $h(\omega)$ by $(h(\omega))_*\nu$, i.e.,
%\[
%\left[(h(\omega))_*\nu \right] (A) = \nu ((h(\omega))^{-1}A), \quad A\in \mathcal B (\Sone),
%\] 
%then $\omega \mapsto \left[(h(\omega))_*\nu \right] (A) $ is a measurable mapping from $\Omega $ to $\mathbb R$ for all $A\in \mathcal B(\Sone)$, where $\mathcal B$
%\end{rem}
%
%
%
%

\begin{rem}\label{rmk:3}
It is  desirable to see if any topological condition  on   perturbations  (in our case, e.g., the continuity of $\theta$) is removable. This boils down to whether the random conjugacy $h:\Omega \to \mathrm{Homeo} (\Sone ,\Sone)$ is measurable without the topological condition.  
However it is unclear to us whether  requiring our topological setup is due to  a substantial obstacle or is an artifact of our construction.
\end{rem}

We need the following proposition for a random Markov partition.
\begin{defn}
Let $f :\Omega \to \mathscr C^r(\Sone ,\Sone)$ be a continuous mapping and denote $f(\omega)$ by $f_\omega$. We say that a finite collection $\{ \mathcal I ^{(\cdot) } _j\} _{j\in \mathcal A}$ of continuous mappings  defined on $\Omega$ with values in $\mathcal F(\Sone)$ is a \emph{Markov partition} of $f$ 
if it satisfies the following conditions:
\begin{itemize}
\item $\mathcal I^\omega _j$ is  a nonempty left-closed and right-open interval
% (i.e.,  of the  form $[a,b)$ with  $a <b$)
  for each $j\in \mathcal A$ and $\omega \in \Omega$,
\item $ \bigsqcup _{j \in \mathcal A} \mathcal I ^\omega _j =\Sone$ for  every $\omega \in \Omega$,
\item $f _\omega (\mathcal I^\omega _j) =\Sone$ and $f_\omega : \mathcal I ^\omega _j \to \Sone$ is a $\Ci ^r$ diffeomorphism  for each $j\in \mathcal A$ and  $\omega \in \Omega$,
\item $(f_\omega )^{-1} (\mathcal I_i ^{\theta \omega}) \cap \mathcal I ^\omega _j$ is a nonempty left-closed and right-open interval for each $i , j\in \mathcal A$ and   $\omega \in \Omega$.
%\footnote{Comment on the fifth condition?} 
\end{itemize}
\end{defn}

\begin{prop}\label{cor:randomMarkov}
Suppose that $0\leq \epsilon <\epsilon _0$. Then there is a Markov partition $\{ \mathcal I _j ^{(\cdot)} \} _{j=0} ^{k-1}$ of $f_\epsilon$ such that for each $0\leq j\leq k-1$, we can find a nonempty open interval $J^\prime$ such that $\mathcal I^\omega _j$ does not intersect $J^\prime$ for  every $\omega \in \Omega$.
\end{prop}
\begin{proof}
 Let $I_j =\pi _{\Sone}([j/k ,(j+1)/k))$ for $0\leq j\leq k-1$. Then $\{ I_j\} _{j=0}^{k-1}$ is a Markov partition of the  $k$-folding map $E_k$. I.e., 
\begin{itemize}
\item $ I_j$ is  a left-closed and right-open interval for each $j\in \mathcal A$,
\item $ \bigsqcup _{j\in \mathcal A}  I  _j =\Sone$,
\item $E_k ( I_j) =\Sone$ and $E_k:  I  _j \to \Sone$ is a $\Ci ^r$ diffeomorphism,
\item $(E_k)^{-1} ( I_i ) \cap I _j$ is a left-closed and right-open interval for each $i , j\in \mathcal A$.
\end{itemize}
We fix $0\leq \epsilon <\epsilon _0$, and let $h:\Omega \to \mathrm{Homeo}(\Sone ,\Sone)$ be the continuous mapping satisfying \eqref{eq:conjugacy} in Theorem \ref{prop:randomshub}. Set $\mathcal I^\omega _j =h(\omega) (I_j)$ for each $\omega \in \Omega$ and $0\leq j\leq k-1$. Then due to Theorem \ref{prop:randomshub}, it is straightforward to see that  $\{ \mathcal I^{(\cdot)} _j\} _{j=0}^{k-1}$ is a Markov partition of $f_\epsilon$ such that 
\[
\sup _{(\omega ,\omega ^\prime) \in \Omega ^2} d_H(\mathcal I_j^\omega ,\mathcal I_j^{\omega ^\prime}) \leq \frac{1}{k}+2\delta _0 <1
\]
for each $0\leq j\leq k-1$, where the last inequality follows from the condition of $\delta _0$ given above \eqref{eq:aaas}. This immediately implies the conclusion.
\end{proof}

\subsection{Coding of graphs}\label{subsection:coding}
Throughout this subsection, we fix $0\leq \epsilon <\epsilon _0$, and let $\{ \mathcal I^{(\cdot)} _j\} _{j=0}^{k-1}$ be the  Markov partition of $f_\epsilon$  constructed in the proof of Proposition \ref{cor:randomMarkov}.
For a word $s =(s_0 ,s_1 ,s_2, \ldots s_{n-1}) \in \mathcal A^n$ and $\omega \in \Omega$, let $ \mathcal I ^\omega _ s$ be a  subset of $\Sone$ defined by
\begin{equation}\label{eq:defofi}
\mathcal I ^\omega  _s  = \bigcap ^{n-1} _{j=0} \left(f^{(j)} _\omega \right)^{-1} \left(\mathcal I  ^{\theta ^j \omega} _{s_j}\right) .
\end{equation}
\begin{lem}\label{lem:1}
For each $n\geq 1$ and  $s  \in \mathcal A^n$,  $\mathcal  I^\omega _s$ is a nonempty left-closed and right-open  interval of $\Sone$ for every $\omega \in \Omega$.
Moreover, the mapping $\omega \mapsto \mathcal I_s ^\omega $ from $\Omega$ to $\mathcal F(\Sone)$ is continuous (in particular, measurable).
\end{lem}
\begin{proof}
Let $n\geq 1$, $s\in\mathcal A^n$ and $\omega \in \Omega$ be given. 
Recall that   $\mathcal I  ^{\theta ^j \omega} _{s_j} = h(\theta ^j \omega) (I_{s_j})$ for each $0\leq j\leq n-1$ by construction. 
Then, 
it follows from Theorem \ref{prop:randomshub} that 
\[
 \left(f^{(j)} _\omega \right)^{-1} \left(\mathcal I  ^{\theta ^j \omega} _{s_j}\right) = h(\omega) \circ \left(E_k^j\right) ^{-1} \left( I_{s_j}\right).
\]
Therefore, due to that $h(\omega)$ is a homeomorphism on $\Sone$, we have
\begin{equation}\label{eq:2rev1}
\mathcal I^\omega _s =h(\omega) \left(I_s\right),\quad \text{where $I_s=\bigcap ^{n-1} _{j=0} \left(E_k^j \right)^{-1} \left( I  _{s_j}\right)$}.
\end{equation}

We next prove that $I  _s$ is a nonempty left-closed and right-open interval (and so is $\mathcal I^\omega _s$ for each $\omega \in \Omega$, due to \eqref{eq:2rev1}) by induction. It is obviously true for $n=1$.  
Assume that the claim is true for a positive integer $n$, i.e., the set $I _s$ is a nonempty left-closed and right-open interval for  every  $s \in \mathcal A^n$. 
Then, for  every   $s =(s_0, s_1, \ldots , s_n) \in \mathcal A^{n+1}$, we have 
\begin{align}\label{eq:23}
I _s &= I_{s_0} \cap \left( \left( E_k\right) ^{-1} \left( \bigcap _{j=1} ^n \left( E_k^{j-1} \right) ^{-1} \left(I _{s_j} \right) \right) \right) \\
&= I _{s_0} \cap \left( \left( E_k \right) ^{-1} \left( I _{\sigma (s)} \right) \right) ,
\end{align}
where $\sigma :\mathcal A^{n+1} \to \mathcal A^n$ is the one-sided shift defined by $\sigma (s) =(s_1, s_2 ,\ldots ,s_n)$ for $s=(s_0, s_1, \ldots ,s_n) \in \mathcal A^{n+1}$. Note that $E_k: I_{s_0}  \to \Sone$ is a diffeomorphism, and  $I _{\sigma (s)}$ is a nonempty  left-closed and right-open  interval by the hypothesis of induction. Hence, $I_s$ is also a nonempty  left-closed and right-open  interval, and we complete the proof of the claim.

We fix $n\geq 1$ and $s\in\mathcal A^n$. We finally prove the continuity of the mapping $\omega \mapsto h(\omega)(I  _s)$
 from $\Omega$ to $\mathcal F(\Sone)$, 
which implies the continuity of  $\omega \mapsto \mathcal I ^\omega _s$ by  \eqref{eq:2rev1}.
%Let $d_\Omega (\cdot ,\cdot)$ be the metric of $\Omega$.
Fix $\omega \in \Omega$, and let $\kappa =\kappa (\omega)>0$ be a real number such that $\kappa <1-\vert h(\omega)(I_s)\vert $.
By the continuity of $h:\Omega \to \mathrm{Homeo}(\Sone , \Sone)$ at $\omega$, if $\omega ^\prime$ is sufficiently  close to $\omega$, then 
\begin{equation}\label{eq:2r567}
d_{\Ci ^0}(h(\omega) , h(\omega ^\prime)) < \min\left\{ \frac{\kappa}{2}, \frac{1-\kappa - \vert h(\omega)(I_s)\vert }{2}
\right\} .
\end{equation}
For each $\omega ^\prime$ satisfying  \eqref{eq:2r567}, it is easy to see that $\vert h(\omega ^\prime )(I_s) \vert < \vert h(\omega )(I_s) \vert +\kappa$ and 
\begin{equation}\label{eq:2r568}
d_{\Ci ^0}(h(\omega), h(\omega ^\prime)) <\min \left\{ \frac{1-\vert h(\omega )(I_s) \vert}{2}, \frac{1 - \vert h(\omega ^\prime)(I_s)\vert }{2}
\right\}.
\end{equation}
Let $\tilde I_s$ be an interval of $\R$ such that $\pi _{\Sone}(\tilde I_s) =I_s$. 
Let $\tilde h(\omega)$ and $\tilde h(\omega ^\prime)$  be the lifts on $\R$ of homeomorphisms $h(\omega)$ and $h(\omega ^\prime)$ on $\Sone$, respectively, such that $d_{\Ci^0}(\tilde h(\omega), \tilde h(\omega ^\prime))<1$. 
Note that $h(\omega)$ and $h(\omega ^\prime)$ are orientation-preserving  by construction in the proof of Theorem \ref{prop:randomshub}, and thus  $\tilde h(\omega)$ and $\tilde h(\omega ^\prime)$ are monotonically increasing.
We let $\zeta$ and $\zeta ^\prime$ denote the midpoints of $\tilde h(\omega)(\tilde I_s)$ and $\tilde h (\omega ^\prime)(\tilde I_s)$, respectively. 
Then, it is straightforward to see that \eqref{eq:2r568} implies $\tilde h(\omega^\prime) (\tilde I_s) \subset [\zeta -1/2, \zeta +1/2)$ and $\tilde h(\omega) (\tilde I_s) \subset [\zeta ^\prime -1/2, \zeta ^\prime+1/2)$. Therefore,
%
%It is straightforward to see that for each nonempty left-closed and right-open interval $I$,  if $h_0$ and $h_1$ are orientation-preserving  homeomorphisms on $\Sone$ 
%%and $d_{\Ci ^0}(h_0 ,h_1) $ is so small 
%such that $h_0(I) \cap h_1(I) \neq \emptyset$, 
%%and that both of $h_0$ and $h_1$ are orientation-preserving or orientation-reversing, 
%then 
 we have
%$h_0(I)$ and $h_1(I)$ are left-closed and right-open intervals having an intersection, and thus
\[
d_H (h(\omega)(I_s), h(\omega ^\prime)(I_s)) = \max_{x\in \partial I_s} d_{\Sone} (h(\omega)(x),h(\omega ^\prime)(x)),
\]
which is bounded by $d_{\Ci ^0}(h(\omega),h(\omega ^\prime))$, where $\partial I_s$ is the boundary of  $I_s$. 
%Therefore, for each $\omega  \in \Omega$, it follows from the continuity of $h:\Omega \to \mathrm{Homeo}(\Sone,\Sone)$ at $\omega $ that if $\omega ^\prime$ is sufficiently close $\omega $, then 
%%with the notation $d_\Omega$ for the metric of   $\Omega$, it follows from the continuity of $\omega \mapsto h(\omega)$ that for any $n\geq 1$ and $s\in \mathcal A^n$, there is a positive number $\delta _{n,s}>0$ such that if $\omega ,\omega ' \in \Omega$ with $d_\Omega (\omega ,\omega ')$, then
%\[
%d_H(h(\omega )(I_s), h(\omega ^\prime)(I_s)) \leq d_{\Ci ^0}(h(\omega ), h(\omega ^\prime)).
%\]
By using the continuity of $h:\Omega \to \mathrm{Homeo}(\Sone,\Sone)$  at $\omega $ again, 
%together with \eqref{eq:2rev1}, 
we obtain the continuity of $h(\cdot)(I_s)$ at $\omega$.  
%which comple
%  
%$\omega \mapsto (f^{(j)} _\omega )^{-1} (\mathcal I  ^{\theta ^j \omega} _{s_j}) 
%$ is the composition of a mapping $\alpha :\Omega \times \Omega \to \mathcal F (\Sone)$ given by 
%\[
%\alpha(\omega ,\omega ^{\prime})= \left(f^{(j)} _\omega \right)^{-1} (\mathcal I  ^{\theta ^j \omega ^\prime} _{s_j}) ,\quad (\omega ,\omega ^\prime ) \in \Omega \times \Omega
%\] 
%and a continuous mapping $\Lambda :\Omega \to \Omega \times \Omega$ given by $\Lambda (\omega)=(\omega ,\omega)$ for $\omega \in \Omega$. 
%Thus, in view of \eqref{eq:defofi}, it suffices to show the continuity of $\alpha :\Omega \times \Omega \to \mathcal F(\Sone)$. 
%
%We fix $n \geq 1$,   $s\in \mathcal A^n$ and  $0\leq j\leq n$. It follows from the continuity of $\Omega \times \Sone \ni (\omega ,x) \mapsto f_\epsilon (\omega ,x)$ and $\theta$ that for each $I \in \mathcal F (\Sone)$, the mapping $\omega \mapsto (f^{(j)} _\omega )^{-1} (I)$ is a continuous mapping from $\Omega $ to $\mathcal F (\Sone)$. Therefore $\omega \mapsto \alpha(\omega ,\omega ^\prime)$ is a continuous mapping from $\Omega$ to $\mathcal F (\Sone)$ for each $\omega ^\prime \in \Omega$.
%On the other hand, 
% the mapping $\Omega \ni \omega ^\prime \mapsto \alpha(\omega ,\omega ^\prime)$ is  continuous for each $\omega \in \Omega$ since $ (f^{(j)} _\omega )^{-1} :\mathcal F (\Sone) \to \mathcal F (\Sone)$ is continuous for each $\omega \in \Omega$ and $\mathcal I^{(\cdot)} _j$ is continuous for each $0\leq j\leq k-1$. 
Since the choice of $\omega \in \Omega$ is arbitrary,  $\omega \mapsto  h(\omega)(I _s)$ is a continuous mapping from $\Omega$ to $\mathcal F(\Sone)$.
\end{proof}

Let $\N _0$ be the set of all nonnegative integers $\{ 0, 1, \ldots\}$.
For a  sequence $s=(s_0,s_1, \ldots , s_m,s_{m+1},\ldots , s_{n},\ldots )\in \mathcal A^{\N _0}$ with integers $n\geq m\geq 0$, 
we  denote   $(s_m,s_{m+1},\ldots , s_{n})$   by $[s]_m^{n}$. 
We define  $\mathcal I_ s^\omega \in \mathcal F(\Sone)$ by
$ 
\mathcal I^\omega _s=\bigcap _{n\geq 0} \mathcal I_{[s]_0^n}^\omega .
$
Then,
since  the inverse branches of $f_\omega $ are contractions due to \eqref{assumption2},  it follows from Lemma \ref{lem:1} that  $\mathcal I_ s^\omega$ 
 is a point set. We denote the point by $X_s (\omega)$. Then, by the continuity in Lemma \ref{lem:1}, the mapping $\omega \mapsto \{X_s(\omega) \} $ from $\Omega$ to $\mathcal F(\Sone)$  is continuous, so is the mapping $X_s:\Omega \to \Sone$. 
 %Furthermore, it is easy to see that for $\mathbb P$-almost every $\omega \in \Omega$ and any $x \in \Sone$, we can find a unique sequence $s =(s_0, s_1,\ldots )\in \mathcal A^{\N _0}$ given by 
%\[
%s_j =i \quad \text{if} \quad f^{(j)} _\omega (x) \in \mathcal I ^{\theta ^j\omega} _{i},
%\]
%and  we  have  $x =X_s (\omega)$.
%\footnote{Can be removed?}

The following  lemmas on the graphs of $X_s$'s are simple but substantial in the proof of Theorem \ref{conj:1}. Let $\sigma :\mathcal A^{\mathbb N_0} \to  \mathcal A^{\mathbb N_0}$ be the one-sided shift given by $\sigma (s) =(s _1,s _2,\ldots)$ for $s =(s _0,s _1,\ldots) \in \mathcal A^{\mathbb N_0}$. 
\begin{lem}\label{lem:22}
For any $s\in \mathcal A^{\N _0}$, we have
\[
f_\omega (X_s(\omega)) =X_{\sigma (s)} (\theta \omega) , \quad \omega \in \Omega,
\]
and we can find a nonempty open interval $J^\prime $ such that $X_s(\omega)$ does not intersect $J^\prime$ for  every $\omega \in \Omega$.
%Furthermore, for $\mathbb P^2$-almost every $(\omega ,\omega ^\prime)$ we have  
%\[
%d_{\Sone} (X_s (\omega) , X_s (\omega ^\prime) ) \leq \frac{1}{k+1}. 
%\]
\end{lem}
\begin{proof}
%Noting that $(f^{(j)} _\omega )^{-1} =(f_\omega )^{-1} \circ (f _{\theta \omega})^{-1} \circ \cdots \circ (f_{\theta ^{j-1}\omega} )^{-1}$, we have
%\begin{align*}
%f_\omega\left(\mathcal I ^\omega  _s  \right) &= \bigcap ^{n-1} _{j=0}f_\omega \circ \left(f^{(j)} _\omega \right)^{-1} \left(\overline{\mathcal I  ^{\theta ^j \omega} _{s_j}}\right) \\
%&=  \left( f_\omega \left(\overline{\mathcal I ^\omega _{s_0}} \right) \right) \cap \left(\bigcap ^{n-1} _{j=1} \left(f^{(j-1)} _{\theta \omega} \right)^{-1} \left(\overline{\mathcal I  ^{\theta ^j \omega} _{s_{j}}}\right) \right) \\
%&=  \left( f_\omega \left(\overline{\mathcal I ^\omega _{s_0}} \right) \right) \cap \left(\bigcap ^{n-2} _{j=0} \left(f^{(j)} _{\theta \omega} \right)^{-1} \left(\overline{\mathcal I  ^{\theta ^j (\theta \omega)} _{s_{j+1}}}\right) \right) \\
%\end{align*}
%For each $s =(s_0, s_1 ,\ldots ) \in \mathcal A ^{\N _0}$, we denote $(s_0, s_1, \ldots ,s_{n-1})$ by $[s]_n$. 
From \eqref{eq:23},  it follows that 
\begin{align*}
f_\omega \left(\mathcal I ^\omega _{[s] _0^n} \right) & =f_\omega \left(\mathcal I^\omega _{s_0}\right) \cap  \mathcal I ^{\theta \omega} _{\sigma ([s] _0^n)}  \\
&= \Sone \cap \mathcal I ^{\theta \omega} _{\sigma ([s] _0^n)} 
= \mathcal I ^{\theta \omega} _{\sigma ([s] _0^n)} .
\end{align*}
Thus, by the virtue of the continuity of $f_\omega :\mathcal F(\Sone) \to \mathcal F(\Sone)$,  we get $f_\omega (\mathcal I ^\omega _s) =\mathcal I _{\sigma (s)} ^{\theta \omega}$ for  every $\omega \in \Omega$. Together with the construction of $X_s $ and Proposition \ref{cor:randomMarkov}, this implies  the conclusion.
\end{proof}

\begin{lem}\label{lem:123}
Let $x$ be a point in $\Sone$ and $s =s(x)=(s_0,s_1,\ldots ) \in \mathcal A^{\N _0}$ the coding of $x$ by $E_k$, i.e., $E_k^j(x) \in I_{s_j}$ for each $j\geq 0$. Then,  $h(\omega)(x) =X_s(\omega)$ for  every $\omega \in \Omega$ and $x\in \Sone$, where $h$ is given in Theorem \ref{prop:randomshub}.
\end{lem}

\begin{proof}
It follows from \eqref{eq:conjugacy} and \eqref{eq:defofi}  that for each $n\geq 1$, we have
\begin{align*}
\mathcal I ^\omega  _{[s]_0^n}  &= \bigcap ^{n-1} _{j=0} \left(f^{(j)} _\omega \right)^{-1} \circ h  (\theta ^j \omega)\left( I _{s_j}\right) \\
&= \bigcap ^{n-1} _{j=0} \left(h  ( \omega) \circ E_k^{-j} ( I _{s_j})\right).
\end{align*}
Since $h(\omega):\Sone \to \Sone$ is a homeomorphism, this implies 
\[
 \mathcal I ^\omega  _s    =h(\omega)\left(\bigcap ^{\infty} _{j=0}  E_k ^{-j}( I _{s_j}) \right).
\]
Recalling $\mathcal I^\omega _s=\{ X_s(\omega) \}$, we 
immediately  get the conclusion. 
%\footnote{CHECK.}
\end{proof}

\subsection{Ergodicity and  SRB property}

As a final preparation before turning to the proof of Theorem \ref{conj:1}, we may need to find a sequence $s^{\prime \prime}$ such that  time averages  along the future random orbit of $X_{s^{\prime \prime}}$ $\mathbb P$-almost surely coincide with their corresponding integrals with respect to a probability measure that is equivalent to the normalised Lebesgue measure, and that the past random orbit of $X_{s^{\prime \prime}}$ is $\mathbb P$-almost surely dense.
We prepare language for this purpose. 

%By abuse of notation, we simply write $\epsilon _0$ for the minimum of 
Fix $0\leq \epsilon <\epsilon _0$ and let $h=h_\epsilon$ be the mapping given in Theorem \ref{prop:randomshub}.
Then, 
$ (\omega ,x) \mapsto h(\omega)(x)$ is a continuous mapping from $\Omega \times \Sone $ to $\Sone$ due to the continuity of $h:\Omega \to \mathrm{Homeo}(\Sone ,\Sone)$. Therefore for each continuous function $\varphi :\Sone \to \R$, the function $\Phi _h:\Omega \times \Sone \mapsto \R$ given by $\Phi _h(\omega ,x) = \varphi (h(\omega )(x))$ for $(\omega ,x) \in \Omega \times \Sone$ is continuous, in particular measurable. 
Moreover, it is straightforward to see that $\Vert \Phi_h\Vert _{L^1_{m\times \mathbb P}}\leq \Vert \varphi\Vert _{\Ci ^0}$, i.e., 
$\Phi _h$ is an integrable function with respect to the product measure $m\times \mathbb P$, where $m$ is the normalised Lebesgue measure on $\Sone$. Thus it follows from Fubini's theorem that $ \omega \mapsto \int \Phi _h(\omega , \cdot )dm$ is measurable. Furthermore, since $\varphi $ is an arbitrary continuous function, it follows from Riesz representation theorem that we can find a probability measure $(h (\omega)) _* m$ on $\Sone$  such that $ \int \Phi _h(\omega , \cdot )dm = \int \varphi  d[ (h (\omega)) _* m]$ for each $\omega \in \Omega$. (This probability measure is called the pushforward of $m$ by $h(\omega)$.)
%\[
%\left(\left(h\left(\omega \right) \right)_* m\right) (A) =m\left( \left( h(\omega) \right) ^{-1} A\right), \quad A\in \mathcal B(\Sone).
%\]

%For a Borel probability measure $\nu$ on $\Sone$ and $\omega \in \Omega$, we denote the pushforward of $\nu$ under $h(\omega) $ by $(h (\omega)) _* \nu$, i.e.,
%\[
%\left(\left(h\left(\omega \right) \right)_* m\right) (A) =m\left( \left( h(\omega) \right) ^{-1} A\right), \quad A\in \mathcal B(\Sone).
%\]
%Then for each $A\in \mathcal B(\Sone)$, the mapping $\omega \mapsto (h (\omega)) _* \nu (A)$ is measurable. Indeed, \footnote{Prove it}

\begin{thm}\label{thm:graphSRB}
For any continuous function $\varphi :\Sone \to \R$, there exist a full measure set $A = A(\varphi)$ of  $(\Sone ,m)$ and a family of full measure sets $\{\Gamma _x\}_{x\in A} =\{\Gamma _x(\varphi)\}_{x\in A}$ of $(\Omega ,\mathbb P)$ such that for each $x \in A$ and $\omega \in \Gamma _x$, if we set  $s^{\prime \prime}=s^{\prime \prime}(x) \in \mathcal A^{\mathbb N_0}$ as the coding of $x$ by $E_k$ given in Lemma \ref{lem:123}, then we have 
\[
\lim _{n\to \infty} \frac{1}{n} \sum ^{n-1} _{j=0} \varphi \circ f_{\omega} ^{(j)} (X_{s^{\prime \prime}}(\omega)) = \int \left(\int \varphi d\left[ \left( h(\cdot) \right) _* m\right] \right)d\mathbb P.
\]
\end{thm}

\begin{proof}
For given  $x\in \Sone$, let $s^{\prime \prime} \in \mathcal A^{\N _0}$ be the coding of $x$ by $E_k$.
By Theorem \ref{prop:randomshub} and Lemma \ref{lem:123},  we have
\[
\varphi \circ f_{\omega} ^{(j)} (X_{s^{\prime \prime}}(\omega)) =\varphi ( f_{\omega} ^{(j)} \circ h(\omega) (x)) =\varphi ( h(\theta ^j \omega) \circ E_k^j (x))
\]
 for all $j\geq 0$ and   $\omega \in \Omega$. 
Thus, if we define 
%a measurable function $\Phi _h: \Omega \times \Sone \to \R$, $\Phi _h (\omega ,x) = \varphi (h(\omega)(x))$ \footnote{Check the measurability.} 
%and 
a measurable mapping $\Theta _{E_k} :\Omega \times \Sone \to \Omega \times \Sone$ of the direct-product form
\[
\Theta _{E_k} (\omega ,x) =(\theta \omega , E_k(x)), \quad (\omega ,x) \in \Omega \times \Sone ,
\]
then
\begin{equation}\label{eq:revis1}
\lim _{n\to \infty} \frac{1}{n} \sum ^{n-1} _{j=0} \varphi \circ f_{\omega} ^{(j)} (X_{s^{\prime \prime}}(\omega)) = \lim _{n\to \infty} \frac{1}{n} \sum ^{n-1} _{j=0} \Phi _h \circ \Theta _{E_k} ^j (\omega ,x),
\end{equation}
where $\Phi _h: \Omega \times \Sone \to \R$ is the integral function given above Theorem \ref{thm:graphSRB}. 
Recall that $\theta$ is assumed to be ergodic, and that $E_k$ is also known to be ergodic (see, e.g., \cite{KH95}).
Hence, applying Birkhoff's ergodic theorem to the ergodic transformation $(\Theta _{E_k}, \mathbb P\times m)$ with the integral function $\Phi _h$, we can find a  $(\mathbb P\times m)$-full measure set $G$ such that
the time average in \eqref{eq:revis1}   coincides with $\int \Phi _h dm d\mathbb P$ for all $(\omega ,x) \in G$.

Let $1 _{G}:\Omega \times \Sone \to \R$ be the indicator function  of $G$.
Then,
it follows from  Fubini's theorem that
there is a full measure set $A_1$ of $(\Sone ,m)$ such that 
$1_{G}(x,\cdot ): \Omega \to \R$ is integrable for  every $x\in A_1$.
Let $\Gamma _x=\mathrm{supp} (1_{G}(x,\cdot))$ and $A=\{ x\in A_1 \mid \mathbb P(\Gamma _x) =\int 1_G(x, \cdot )d\mathbb P =1\}$. 
Then $m(A)=1$ (otherwise, it contracts to that $\mathbb P\times m(G)=1$),
%we get a family of measurable functions $\Gamma x\mapsto \int 1_G(\cdot ,x)d\mathbb P$
and the conclusion immediately follows from \eqref{eq:revis1} 
%these estimates 
and the observation above Theorem \ref{thm:graphSRB}.
\end{proof}

\begin{rem}\label{rmk:4}
Theorem \ref{thm:graphSRB}  is exactly the reason why Theorem \ref{conj:1}  is only stated for $\mathbb P$-almost every $\omega \in \Omega$ (not for all $\omega \in \Omega$). 
See that  the statements of Proposition \ref{randomDowker}, Proposition \ref{cor:randomMarkov}, Lemma \ref{lem:22} and Theorem \ref{thm:graphSRBbackward} are given for all $\omega \in \Omega$, while Theorem \ref{thm:graphSRB} is only stated for $\mathbb P$-almost every $\omega$. 
This restriction in Theorem \ref{thm:graphSRB} comes from that the convergence of  time averages along the future random orbit of $X_{s^{\prime \prime}}$  is ensured by  the ergodicity of $(\Theta _{E_k} ,\mathbb P\times m)$, ultimately, by the ergodicity of $(\theta ,\mathbb P)$. 
%Therefore,  it seems to be impossible to replace  ``for $\mathbb P$-almost every $\omega$'' with ``for all $\omega$''  in Theorem \ref{thm:graphSRB}, if one employs the techniques in this paper.
\end{rem}

Theorem \ref{thm:graphSRB} tells that time averages along  the future random orbit of $X_{s^{\prime \prime}}$ (with some $s^{\prime \prime} \in \mathcal A^{\N _0}$) $\mathbb P$-almost surely coincide with their corresponding integrals with respect to a probability measure that is equivalent to $\mathbb P\times m$.
However, this may not imply in general that the \emph{past} random orbit of $X_{s^{\prime \prime}}$ is $\mathbb P$-almost surely dense in $\Sone$, 
and we need the following theorem  to apply Proposition \ref{randomDowker}.

%\begin{rem}
%Theorem \ref{thm:graphSRB} together with Theorem \ref{thm:BKS} gives  information on the geometry of the set of points for which \eqref{eq:stability} holds. That is, it is easy to see that a full measure set $\{ (\omega , h(\omega)(x)) \mid (\omega ,x) \in \Gamma \times A \}$ is included in the set of points satisfying \eqref{eq:stability} (recall Lemma \ref{lem:123}). 
% Furthermore, these results lead to  a  formula for the disintegration of the unique aceip $\mu ^\epsilon$: we  have
%% \footnote{Fix the crash of notation $h$ with BKS's. %More detailed comment on the set satisfying the SRB condition?
%\[
%\mu _\omega ^\epsilon=\left(h\left(\omega \right)\right) _* m, \quad \text{$\mathbb P$-almost surely}.
%\]
%\end{rem}

\begin{thm}\label{thm:graphSRBbackward}
There exists a full measure set $A$ of  $(\Sone ,m)$ 
 such that for each $x \in  A$, if we set  $s^{\prime \prime}=s^{\prime \prime}(x) \in \mathcal A^{\mathbb N_0}$ as the coding of $x$ by $E_k$ given in Lemma \ref{lem:123}, then we have 
\begin{equation}\label{eq:revis3}
\lim _{n\to \infty} \frac{1}{n} \sum ^{n-1} _{j=0} \varphi \circ f_{\theta ^{-j} \omega} ^{(j)} (X_{s^{\prime \prime}}(\theta ^{-j}\omega)) = \int \varphi d\left[ \left( h(\omega) \right) _* m\right] ,
\end{equation}
for any continuous function $\varphi :\Sone \to \R$ and $\omega \in \Omega$.
%, where $A$ is independent of $\varphi$ and $\omega$.

Furthermore, for any $s^{\prime \prime} =s^{\prime \prime}(x)$ with $x\in A$, the past random orbit of $X_{s^{\prime \prime}}$ is  dense at every $\omega \in \Omega$.
\end{thm}

\begin{proof}
We first recall that  the normalised Lebesgue  measure $m$ is the unique SRB probability measure of $E_k$. I.e., there exists a full measure set $A\subset \Sone$ such that for every $x\in A$ and   continuous function $\psi :\Sone \to \R$, we have
\begin{equation}\label{eq:revis2}
\lim _{n\to \infty} \frac{1}{n} \sum ^{n-1} _{j=0} \psi ( E_k^j (x)) = \int \psi dm.
\end{equation}
A standard reference for this property  is \cite{Viana99}.

Let   $\varphi :\Sone \to \R$ be a continuous function and  $\omega \in \Omega$. 
In a similar manner to the proof of Theorem \ref{thm:graphSRB}, we have
\[
\varphi \circ f_{\theta ^{-j}\omega} ^{(j)} (X_{s^{\prime \prime}}(\theta ^{-j}\omega)) =\varphi (h(\omega) \circ E_k^j (x))
\]
 for all $j\geq 0$. Thus,
\[
\lim _{n\to \infty} \frac{1}{n} \sum ^{n-1} _{j=0} \varphi \circ f_{\theta ^{-j} \omega} ^{(j)} (X_{s^{\prime \prime}}(\theta ^{-j}\omega)) =\lim _{n\to \infty} \frac{1}{n} \sum ^{n-1} _{j=0} \varphi \circ h(\omega) (E_k^j (x)) .
\]
Therefore, applying \eqref{eq:revis2} with $\psi = \varphi \circ h(\omega)$, we have
\[
\lim _{n\to \infty} \frac{1}{n} \sum ^{n-1} _{j=0} \varphi \circ f_{\theta ^{-j} \omega} ^{(j)} (X_{s^{\prime \prime}}(\theta ^{-j}\omega)) = \int  \varphi \circ h(\omega) dm = \int \varphi d\left[ \left( h(\omega) \right) _* m\right].
\]

The later assertion is shown by contradiction. Fix $s^{\prime \prime}=s^{\prime \prime}(x)$ with some  $x\in A$ and assume that the past random orbit of $X _{s^{\prime \prime}}$ is not dense at some $\omega \in \Omega$.
Then, there is a nonempty open interval $J\subset \Sone$ with which the past random orbit $\{ f^{(j)}_{\theta ^{-j}\omega} (X_{s^{\prime \prime}}(\theta ^{-j}\omega))\} _{j\geq 0}$ does not  intersect. Thus if we let $\varphi$ be a continuous nonnegative function whose support is a nonempty set  included in $J$, then the left-hand side of  \eqref{eq:revis3} is equal to zero, while the right-hand side is nonzero since $h(\omega):\Sone \to \Sone$ is a homeomorphism so that the  support of $(h(\omega))_*m$ coincides with $\Sone$.  This is a contradiction, and we complete the proof.
\end{proof}

Since the intersection of two full measure sets  is also a full measure set,  we immediately get the following theorem  by  Theorem \ref{thm:graphSRB} and  \ref{thm:graphSRBbackward}.
\begin{thm}\label{thm:graphSRBfull}
For any continuous function $\varphi :\Sone \to \R$, 
there exist a full measure set $A = A(\varphi)$ of  $(\Sone ,m)$ and a family of full measure sets $\{\Gamma _x\}_{x\in A} =\{\Gamma _x(\varphi)\}_{x\in A}$ of $(\Omega ,\mathbb P)$ such that for each $x \in A$ and $\omega \in \Gamma _x$,
%there exist full measure sets  $\Gamma = \Gamma (\varphi)$ and $A = A(\varphi)$ of $(\Omega ,\mathbb P)$ and $(\Sone ,m)$,  respectively, such that for each $x \in  A$, 
if we set  $s^{\prime \prime}=s^{\prime \prime}(x) \in \mathcal A^{\mathbb N_0}$ as the coding of $x$ by $E_k$ given in Lemma \ref{lem:123}, then 
%for every $\omega \in \Gamma$, 
we have that
\begin{enumerate}
\item[$\mathrm{(1)}$] $\displaystyle \lim _{n\to \infty} B_n(\varphi ;\omega , X_{s^{\prime \prime}}(\omega)) = \int \left(\int \varphi d\left[ \left( h(\cdot) \right) _* m\right] \right)d\mathbb P$;
\item[$\mathrm{(2)}$] the past random orbit of $X_{s^{\prime \prime}}$ at $\omega$ is  dense in $\Sone$.
\end{enumerate}
%
%\begin{alignerate}
%\item & \displaystyle \lim _{n\to \infty} B_n(\varphi ;\omega , X_{s^{\prime \prime}}(\omega)) = \int \left(\int \varphi d\left[ \left( h(\cdot) \right) _* m\right] \right)d\mathbb P;\\
%\item & \text{the past random orbit of $X_{s^{\prime \prime}}$ at $\omega$ is  dense.
%}
%\end{alignerate}
\end{thm}

%Finally, we need the following theorem to ensure that $s^{\prime \prime}$

\subsection{The end of the proof}
We will construct a sequence $\bar s\in \mathcal A^{\N_0}$ such that $X_{\bar s}$ $\mathbb P$-almost surely satisfies the hypotheses in Proposition \ref{randomDowker}.
As in the treatment undertaken by Takens \cite[Section 4]{Takens2008}, this sequence will be an appropriate  combination of a periodic sequence $s^\prime$ and a sequence $s^{\prime \prime}$ generating a probability measure that is equivalent to the normalised Lebesgue measure.

%In order to decide the proper lengths of the sequences, we first prepare several lemmas.
Fix $0\leq \epsilon <\epsilon _0$ and let $h=h_\epsilon$ be the mapping given in Theorem \ref{prop:randomshub}. 
Let $s^\prime \in \mathcal A^ {\N_0}$ be a periodic sequence. For simplicity, we set  $s^\prime = (00\ldots )$.  
 By virtue of Proposition \ref{cor:randomMarkov} and Lemma \ref{lem:22}, there exist nonempty open intervals $J\subset J^\prime \subset \Sone$ both of which do not intersect $\mathcal I ^\omega _0$ (in particular, $X_{s^\prime} (\omega)=X_{(00\ldots)} (\omega)$) for  every $\omega$, where the inclusion $J\subset J^\prime $ is strict.  
Let $\varphi _0:\Sone \to \R$ be a $\Ci ^1$  function such that 
\begin{itemize}
\item the support of $\varphi _0$ is included in $J^\prime$,
\item $ \varphi _0(x) =1$ for all $x\in J$,
\item $0\leq \varphi _0(x) \leq 1$ for all $x\in \Sone$. 
\end{itemize}
Then, for all $n\geq 1$ and $\omega \in \Omega$, we have
\begin{equation}\label{eq:sss5}
B_{n} (\varphi _0 ;\omega , X_{s^{\prime }}(\omega))=0 .
\end{equation}
Moreover, it follows from Theorem   \ref{thm:graphSRBfull}  that we can find a sequence $s^{\prime \prime} \in \mathcal A^{\N _0}$ 
%(i.e., $s^{\prime \prime} = s^{\prime \prime}(x)$ with   some $x\in A(\varphi _0)$) 
and a  full measure set $\tilde \Gamma =\tilde \Gamma _{s^{\prime\prime}} (\varphi _0)$ of $(\Omega, \mathbb P)$ such that 
$B_{n}(\varphi _0 ;\omega ,X_{ s^{\prime \prime}}(\omega))$  converges to $\int (\int \varphi _0d[(h(\cdot ))_*m] )d\mathbb P$ as $n$ goes to infinity for  every  $\omega \in \tilde \Gamma$.
%For $s =(s_0, s_1\ldots s_{m-1}, s_m, s_{m+1},\ldots ,s_{n-1}, \ldots )\in \mathcal A^{\N _0}$ with integers $n> m\geq 0$, we simply write $[s]_m^n$ for $(s_m, s_{m+1},\ldots ,s_{n-1})$. \footnote{CHECK if $n>m$ is OK. MOVE to where $[s]_n$ was defined.}
Let $\Gamma = \cap _{n\geq 0} \theta ^{-n} (\tilde \Gamma)$, then it is straightforward to see that $\theta(\Gamma) \subset \Gamma \subset \tilde \Gamma$ and $\mathbb P(\Gamma) =1$.\footnote{Assume that $\mathbb P(\theta ^{-1}(\tilde \Gamma) \cap \tilde \Gamma) <1$. Then $\mathbb P(\theta ^{-1} \tilde \Gamma \cup \tilde \Gamma) =\mathbb P(\theta ^{-1} (\tilde \Gamma)) +\mathbb P(\tilde \Gamma) -\mathbb P(\theta ^{-1}(\tilde \Gamma) \cap \tilde \Gamma)>1$ since $\theta$ is measure-preserving and $\mathbb P(\tilde \Gamma)=1$. This contradicts to that $\mathbb P$ is a probability measure, so that $\mathbb P(\theta ^{-1}(\tilde \Gamma) \cap \tilde \Gamma)=1$. Reiterating the argument, we get that $\cap _{0\leq n\leq N} \theta ^{-n}(\tilde \Gamma)$ is a full measure set for each $N\geq 0$, which immediately implies $\mathbb P(\Gamma) =1$.}
For convenience, we define a nonnegative number $\rho _n (\omega)$ with $n\geq 1$ and $\omega \in \Omega$ by
\begin{equation}\label{eq:revisa7}
\rho _n(\omega) = \left\vert B_{n} (\varphi _0 ;\omega , X_{s^{\prime \prime}}(\omega)) -\int \left(\int \varphi _0d[(h(\cdot ))_*m] \right)d\mathbb P \right\vert .
\end{equation}
Then $\rho _n(\theta ^N\omega)$ converges to zero as $n$ goes to infinity  for every $\omega \in \Gamma$ and $N\geq 0$.

Recall that for a sequence $s= (s_0, s_1,\ldots ,s_m, s_{m+1},\ldots ,s_{n},\ldots )\in \mathcal A^{\N_0}$ with some integers $n\geq m\geq 0$, $(s_m, s_{m+1}, \ldots ,s_{n})$ is denoted by $[s]^n_m$. For a positive  number $a$, we write $[a]$ for the integer part of $a$.
\begin{lem}\label{lem:rev1}
For any integer $m\geq 0$ and real number $\rho >0$, we can find an integer $n\geq 2m+2$ such that,
for all $\omega \in \Omega$ and  $s, t\in \mathcal A^{\N _0}$ satisfying that $[s]_m^{n-1}=[t]^{n-1}_m$, we have 
\begin{equation}\label{eq:s1}
 \left\vert B_{\left[\frac{n}{2}\right]} (\varphi _0 ;\omega , X_{s}(\omega))  -  B_{\left[\frac{n}{2}\right]}  (\varphi _0 ;\omega , X_{t}(\omega)) \right\vert \leq \rho .
\end{equation}
\end{lem}
\begin{proof}
We  fix $m\geq 0$ and $\rho >0$. Note that for any $n^\prime \geq m+1$, $\omega \in \Omega$ and $x\in \Sone$, we have
\begin{equation}\label{eq:revisb1}
B_{n^\prime } (\varphi _0;\omega,x) = \frac{m}{n^\prime }   B_m(\varphi _0;  \omega , x) +\frac{n^\prime -m}{n^\prime}  B_{n^\prime -m}(\varphi _0; \theta ^m \omega , f^{(m)}_\omega (x))
\end{equation}
and 
\begin{equation}\label{eq:revisb2}
\left\vert  \frac{m}{n^\prime }   B_m(\varphi _0;  \omega , x) \right\vert \leq \frac{m}{n^\prime } \Vert \varphi _0\Vert _{\Ci ^0} \leq \frac{m}{n^\prime}.
\end{equation} 
Hence, 
if we take $n\geq 2m+2$ sufficiently large, then $\left[\frac{n}{2}\right] \geq m + 1$ and 
%(by applying \eqref{eq:revisb1} and \eqref{eq:revisb2} with $n^\prime =\left[\frac{n}{2}\right]$) 
we get
\begin{multline}\label{eq:la1}
 \left\vert B_{\left[\frac{n}{2}\right]} (\varphi _0 ;\omega , X_{s}(\omega))  -  B_{\left[\frac{n}{2}\right]}  (\varphi _0 ;\omega , X_{t}(\omega)) \right\vert \\
 \leq \frac{\rho}{2} + \left\vert B_{\left[\frac{n}{2}\right]-m}\left(\varphi _0; \theta ^m \omega , f^{(m)}_\omega ( X_{s}(\omega))\right) - B_{\left[\frac{n}{2}\right] -m}\left(\varphi _0; \theta ^m \omega , f^{(m)}_\omega ( X_{t}(\omega))\right) \right\vert  .
\end{multline}

Let $s, t\in \mathcal A^{\N _0}$ satisfying that $[s]_m^{n-1}=[t]_m^{n-1}$ with some $n\geq 2m+2$. Then  by the argument in the proof of  Lemma \ref{lem:22},
$
f^{(m+\ell)}_\omega (X_s(\omega)) = f^{(\ell)} _{\theta ^m \omega}(X_{\sigma ^m s}(\theta ^m \omega))$ and  $f^{(m+\ell)}_\omega (X_t(\omega)) =f^{(\ell)} _{\theta ^m \omega}(X_{\sigma ^m t}(\theta ^m \omega))$,  both of which are in the interval  
$f^{(\ell)} _{\theta ^m \omega} \left(\mathcal I _{[s]_m^{n-1}}^{\theta ^m\omega} \right)
 %=f^{(\ell)} _{\theta ^m \omega}\left( \mathcal I _{[t]_m^{n}}^{\theta ^m\omega}\right)
 = \mathcal I^{\theta ^{m+\ell}\omega}_{\sigma ^\ell [s]_m^{n-1}}$ for each $0\leq \ell \leq \left[\frac{n}{2}\right] -m-1$. Moreover, when $0\leq \ell \leq \left[\frac{n}{2}\right] -m-1$, due to \eqref{assumption2} and \eqref{eq:defofi}, the
 length of the interval $\mathcal I^{\theta ^{m+\ell}\omega}_{\sigma ^\ell [s]_m^{n-1}}$    is  less than 
 $
 \lambda ^{-(n-m-\ell)} \leq \lambda ^{-\frac{n}{2}}.
 $
 (Notice that  the length of  the word $\sigma ^\ell [s]_m^{n-1}$ is $n-m-\ell$.)
% and $\left[\frac{n}{2}\right] -m-1\leq \frac{n}{2}-m$
Thus, it follows from the mean value theorem that 
\[
\left\vert \varphi _0\left( f^{(m+\ell)}_\omega (X_s(\omega)) \right) - \varphi _0\left( f^{(m+\ell)}_\omega (X_t(\omega))\right) \right\vert \leq \lambda ^{-\frac{n}{2}} \Vert \varphi _0 \Vert _{\Ci ^1}
\]
for every $0\leq \ell \leq \left[\frac{n}{2}\right] -m-1$, and we have
\begin{multline}\label{eq:revisa6}
\left\vert B_{\left[\frac{n}{2}\right]-m}\left(\varphi _0; \theta ^m \omega , f^{(m)}_\omega ( X_{s}(\omega))\right) - B_{\left[\frac{n}{2}\right] -m}\left(\varphi _0; \theta ^m \omega , f^{(m)}_\omega ( X_{t}(\omega))\right) \right\vert \\
\leq\left(\left[\frac{n}{2}\right]-m\right) \lambda ^{-\frac{n}{2}} \Vert \varphi _0 \Vert _{\Ci ^1}.
 %\leq \sum _{\ell =0}^{\left[\frac{n}{2}\right]-m -1} \left\vert \varphi _0\left( f^{(\ell)} _{\theta ^m \omega} \left(\mathcal I _{[s]_m^{n}}^{\theta ^m\omega} \right)\right) - \varphi _0\left( f^{(\ell)} _{\theta ^m \omega} \left(\mathcal I _{[t]_m^{n}}^{\theta ^m\omega} \right)\right) \right\vert 
\end{multline}
 %is bounded by $\left(\left[\frac{n}{2}\right]-m\right) \lambda ^{-(\frac{n}{2}-1)} \Vert \varphi _0 \Vert _{\Ci ^1}$. This
 This should be less than $\frac{\rho}{2}$ by taking $n$ sufficiently large, and the conclusion immediately follows from   \eqref{eq:la1}.
%\footnote{Maybe, the continuity if $B_n$ is not so trivial? Use the lemma in NTW? .. Maybe, need to set the staring assumption such that $\mathcal I^\omega _j$ does not intersect $J^\prime$}
\end{proof}

Now we inductively set an increasing sequence  $\{ N_j\}_{j\geq 0}$ of nonnegative integers. Let $\{ \tilde \rho _j\} _{j\geq 1}$ be a sequence  of real numbers in $(0, 1]$ such that  $\tilde \rho _j$ converges to zero as $j$ goes to infinity.  
 Let $N_0 =0$. 
For a given integer $N_{j-1}$ with odd integer $j\geq 1$, we  take  an integer $N_j\geq 6N_{j-1}/\tilde \rho _j+2$ such that if  $s\in \mathcal A^{\N _0}$ satisfies that $[s]_{N_{j-1}}^{N_{j}-1} =[s^\prime]_{0}^{N_{j}-N_{j-1}-1}$, then
\begin{equation} \label{eq:sss6}
 \left\vert B_{\left[N_j/2\right]}  (\varphi _0 ;\omega , X_{s}(\omega))  -  B_{\left[N_j/2\right]-N_{j-1}} (\varphi _0 ;\theta ^{N_{j-1}}\omega , X_{s^\prime}(\theta ^{N_{j-1}}\omega)) \right\vert \leq  \tilde \rho _j, 
\end{equation}
for all $\omega \in \Omega$.
(We can indeed find such $N_j$:~For each $\omega \in \Omega$, $\tilde s \in \mathcal A^{N_{j-1}}$ and $s\in \mathcal A^{\N _0}$ satisfying that $[s]_{N_{j-1}}^{N_{j}-1} =[\tilde ss^\prime]_{N_{j-1}}^{N_{j}-1}$($=[s^\prime]_{0}^{N_{j}-N_{j-1}-1}$), where $\tilde s s^\prime =(\tilde s_0 \tilde s_1 \ldots \tilde s_{N_{j-1}-1} s^\prime _0 s^\prime _1\ldots)$ is the concatenation of $\tilde s=(\tilde s_0 \tilde s_1 \ldots \tilde s_{N_{j-1}-1})$ and $s^\prime$, we have 
\begin{equation} \label{eq:sss6alt}
 \left\vert B_{\left[N_j/2\right]}  (\varphi _0 ;\omega , X_{s}(\omega))  -  B_{\left[N_j/2\right]} (\varphi _0 ;\omega , X_{\tilde s s^\prime}(\omega)) \right\vert \leq  \tilde \rho _j/3, 
\end{equation}
by applying Lemma \ref{lem:rev1} with $\rho =\tilde \rho _j/3$, $m=N_{j-1}$, $n=N_j$ and $t=\tilde s s^\prime$, and with the notation $\xi _j= 1- N_{j-1}/[N_j/2]$, we also have
%we also have that
 \begin{equation} \label{eq:sss6b}
 \left\vert  B_{\left[N_j/2\right]} (\varphi _0 ;\omega , X_{\tilde s s^\prime}(\omega)) -\xi _j\cdot B_{\left[N_j/2\right]-N_{j-1}} (\varphi _0 ;\theta ^{N_{j-1}}\omega , X_{s^\prime}(\theta ^{N_{j-1}}\omega)) \right\vert \leq \tilde \rho _j/3
\end{equation}
with $\vert 1-\xi _j\vert \leq \tilde \rho _j/3$,
by applying \eqref{eq:revisb1} and \eqref{eq:revisb2} with $n^\prime =[N_j/2]$, $m=N_{j-1}$ and $x=X_{\tilde s s^\prime }(\omega)$, together with Lemma \ref{lem:22} and the fact that $N_{j-1}/[N_j/2]\leq \tilde \rho _j/3$.)
%,  after increasing $N_j$  if necessary.)
On the other hand, for a given integer $N_{j-1}$ with even integer $j\geq 1$, we  take  an integer $N_j\geq 6N_{j-1}/\tilde \rho _j+2$ such that if  $s\in \mathcal A^{\N _0}$ satisfies that  $[s]_{N_{j-1}}^{N_{j}-1} =[s^{\prime \prime}]_{0}^{N_{j}-N_{j-1}-1}$, then
\begin{equation}\label{eq:sss7}
 \left\vert B_{\left[N_j/2\right]} (\varphi _0 ;\omega , X_{s}(\omega))  -  B_{\left[N_j/2\right]-N_{j-1}} (\varphi _0 ;\theta ^{N_{j-1}}\omega , X_{s^{\prime \prime}}(\theta ^{N_{j-1}}\omega)) \right\vert \leq \tilde \rho _j,
\end{equation}
for all $\omega \in \Omega$.
% \begin{equation} \label{eq:sss7b}
% \left\vert  B_{\left[N_j/2\right]} (\varphi _0 ;\omega , X_{\tilde s s^{\prime \prime}}(\omega)) -B_{\left[N_j/2\right]-N_{j-1}} (\varphi _0 ;\theta ^{N_{j-1}}\omega , X_{s^{\prime \prime}}(\theta ^{N_{j-1}}\omega)) \right\vert \leq \tilde \rho _j/2. 
%\end{equation}
Finally, we set a sequence $\bar s \in \mathcal A^{\N _0}$ by
\[
[\bar s]_{N_{j-1}}^{N_j-1} =
\left\{\begin{array}{ll}
[s^{\prime}]_{0}^{N_j-N_{j-1}-1} & (j: \text{odd})\\ \relax
[s^{\prime \prime}]_{0}^{N_j-N_{j-1}-1}& (j:\text{even}) .
\end{array}\right.
\]

It follows from \eqref{eq:sss5} and \eqref{eq:sss6}   that for any  $\omega \in \Omega$ and odd integer $j \geq 1$, we have
\[
 \left\vert B_{\left[N_j/2\right]} (\varphi _0 ;\omega , X_{\bar s}(\omega))  - 0 \right\vert \leq \tilde \rho _j.
\]
Hence for any $\omega \in \Omega$, 
\[
%\liminf _{n\to \infty} B_n(\varphi _0;\omega , X_{\bar s}(\omega)) \leq 
\lim _{\ell \to\infty} B_{\left[N_{2\ell +1}/2\right]} (\varphi _0;\omega , X_{\bar s}(\omega)) =0.
\]
Furthermore, it follows from  \eqref{eq:revisa7} and  \eqref{eq:sss7}     that, for any  $\omega \in \Omega$ and  even integer $j\geq 1$, 
\begin{multline*}
 \left\vert B_{\left[N_j/2\right]} (\varphi _0 ;\omega , X_{\bar s}(\omega))  -\int \left(\int \varphi _0d[(h(\cdot ))_*m] \right)d\mathbb P  \right\vert \\
  \leq \tilde \rho _j + \rho _{\left[N_j/2\right]-N_{j-1}}(\theta ^{N_{j-1}}\omega).
\end{multline*}
Hence, due to that $[N_j/2] -N_{j-1}\geq N_{j-1}$ and $\theta (\Gamma)\subset \Gamma$,  for any $\omega \in \Gamma$ we have
\[
%\limsup _{n\to \infty} B_n(\varphi _0;\omega , X_{\bar s}(\omega)) \geq 
\lim _{\ell \to\infty} B_{\left[N_{2\ell}/2\right]}(\varphi _0;\omega , X_{\bar s}(\omega)) =\int \left(\int \varphi _0d[(h(\cdot ))_*m] \right)d\mathbb P > 0.
\]
The last inequality holds because  the support of $h(\omega) _*m$ coincides with $\Sone$ for each $\omega$.
%we have that $\omega \in \Gamma _{N_j}(2^{-j})$. OK? $N_j$ may depend on $\omega$, and it is problem?
Consequently, $X_{\bar s}$ has historic  behaviour for $\mathbb P$-almost every $\omega \in \Omega$. 

At the end, we show that the past random orbit of $X_{\bar s}$  is $\mathbb P$-almost surely dense.
Let $j$ be an even integer and fix $\omega \in \Gamma$. We denote $\sigma ^{N_{j-1}} \bar s$  by $\bar s^j$. 
Then, for each $0\leq \ell \leq N_{j-1}$,
\[
f^{(N_{j-1}+\ell)}_{\theta ^{-(N_{j-1}+\ell)}\omega}\left(X_{\bar s}\left (\theta ^{-(N_{j-1}+\ell)}\omega\right)\right)= f_{\theta ^{-\ell}\omega}^{(\ell)}\left(X_{\bar s^j}\left(\theta ^{-\ell}\omega\right)\right)
\] 
by 
%We let $\sigma ^{N_{j-1}} \bar s$ denote by $t(j)$. 
%Then, by
  Lemma \ref{lem:22}, and
 $
\left[\bar s^j\right]_{\ell}^{N_j-N_{j-1}-1} 
= \left[   s^{\prime \prime}\right]_{\ell}^{N_j-N_{j-1}-1}  .
$
%(Note that $N_j -N_{j-1}-1\geq N_{j-1}$.)
    %and construction of $X_s(\omega)$,  b
Thus, by using Lemma \ref{lem:22} again,  one can see that both  $ f_{\theta ^{-\ell}\omega}^{(\ell)}(X_{\bar s^j}(\theta ^{-\ell}\omega))$ and $f_{\theta ^{-\ell}\omega}^{(\ell)} (X_{s^{\prime \prime}} (\theta ^{-\ell} \omega))$ 
%and $x_2:=f^{(N_{j-1})}_{\theta ^{-N_{j-1}}\omega}(X_{s^{\prime \prime}} (\theta ^{-N_{j-1}}\omega))$ 
are in 
%$\mathcal I ^\omega _{\sigma ^{N_{j-1}}s^{\prime \prime}} $. On the other hand, it follows from that 
$
\mathcal I ^\omega _t
$ with a word $t=\left[   s^{\prime \prime}\right]_{\ell}^{N_j-N_{j-1}-1}$ of length $N_j-N_{j-1}-\ell$. 
Therefore, noting that $N_j-N_{j-1}-\ell \geq  N_j$ for each $0\leq \ell \leq N_{j-1}$, together with \eqref{assumption2} and \eqref{eq:defofi}, we get
\[
\left \vert  f_{\theta ^{-\ell}\omega}^{(\ell)}(X_{\bar s^j}(\theta ^{-\ell}\omega)) - f_{\theta ^{-\ell}\omega}^{(\ell)} (X_{s^{\prime \prime}} (\theta ^{-\ell} \omega))\right\vert \leq C_\omega \lambda ^{-(N_j-N_{j-1}-\ell)}\leq C_\omega \lambda ^{-N_{j-1}},
\]
where $C_\omega >0$ is the maximum of $\vert \mathcal I _j^\omega\vert$ over $0\leq j\leq k-1$. 
%positive constant depending only on the random Markov partition $\{\mathcal I _j^\omega \}_{j=0}^{k-1}$ at $\omega$.
%%On the other hand, the
% %length of $\mathcal I ^\omega _t$ 
%  whose length does not exceed $\lambda ^{-(N_j-N_{j-1})}\leq \lambda ^{-N_{j-1}}$ due to \eqref{assumption2}, \eqref{eq:defofi} and the fact that $t\in \mathcal A^{N_j-N_{j-1}}$. 
%Therefore, 
%%it follows from $\lambda ^{-(N_j-N_{j-1})}$ (by Lemma \ref{lem:rev1}) that
%%  for some positive constant $C_\omega$  only depending on the homeomorphism $h(\omega):\Sone \to \Sone$. 
%%Therefore, 
%the distance between $x_1$ and $x_2$ can be arbitrary small as taking $j$ sufficiently large.
%%(see the proof of Lemma \ref{lem:1}).
%
%
%
%Recall that
Since the past random orbit of $X_{s^{\prime \prime}}$ is $\mathbb P$-almost surely dense on $\Sone$ by Theorem \ref{thm:graphSRBfull}, we immediately conclude that the past random orbit of $X_{\bar s}$  is $\mathbb P$-almost surely dense on $\Sone$. This completes the proof of Theorem  \ref{conj:1} by   Proposition \ref{randomDowker}.

\section*{Acknowledgments}
The author would like to express his deep gratitude to Shin Kiriki and Teruhiko Soma for many fruitful discussions and valuable comments, without which this paper would not have been possible.
The author also greatly appreciates helpful suggestions by Hiroki Takahashi.
Finally the author is  deeply grateful to the anonymous referee for
the many suggestions, all of which substantially improved the paper.

%For precise definition and general properties of a random Markov partition, we refer to \cite{BBM-D}.
\bibliographystyle{my-amsplain-nodash-abrv-lastnamefirst-nodot}
\bibliography{MATHabrv,KNS}

\begin{bibdiv}
\begin{biblist}

\bib{Araujo2000}{article}{
      author={Ara\'ujo, V^^c4^^b1tor},
       title={Attractors and time averages for random maps},
        date={2000},
     journal={Annales de l'Institut Henri Poincare (C) Non Linear Analysis},
      volume={17},
      number={3},
       pages={307\ndash 369},
}

\bib{BBM-D}{inproceedings}{
      author={Baladi, Viviane},
      author={Benedicks, Michael},
      author={Maume-Deschamps, V{\'e}ronique},
       title={Almost sure rates of mixing for iid unimodal maps},
        date={2002},
   booktitle={Annales scientifiques de l'ecole normale sup{\'e}rieure},
      volume={35},
       pages={77\ndash 126},
}

\bib{BS2000}{article}{
      author={Barreira, Luis},
      author={Schmeling, J{\"o}rg},
       title={Sets of ``non-typical'' points have full topological entropy and
  full hausdorff dimension},
        date={2000},
     journal={Israel Journal of Mathematics},
      volume={116},
      number={1},
       pages={29\ndash 70},
}

\bib{BDV04}{book}{
      author={Bonatti, Christian},
      author={D\'iaz, Lorenzo},
      author={Viana, Marcelo},
       title={Dynamics beyond uniform hyperbolicity: A global geometric and
  probabilistic perspective},
   publisher={Springer},
        date={2004},
}

\bib{HK1990}{article}{
      author={Hofbauer, Franz},
      author={Keller, Gerhard},
       title={Quadratic maps without asymptotic measure},
        date={1990},
     journal={Communications in mathematical physics},
      volume={127},
      number={2},
       pages={319\ndash 337},
}

\bib{KH95}{book}{
      author={Katok, Anatole},
      author={Hasselblatt, Boris},
       title={Introduction to the modern theory of dynamical systems},
   publisher={Cambridge University Press},
        date={1995},
}

\bib{KLS2016}{article}{
      author={Kiriki, Shin},
      author={Li, Ming-Chia},
      author={Soma, Teruhiko},
       title={{Geometric Lorenz flows with historic behavior}},
        date={2016},
     journal={Discrete Cont. Dynam. Systems},
      volume={36},
       pages={7021\ndash 7028},
}

\bib{KS2017}{article}{
      author={Kiriki, Shin},
      author={Soma, Teruhiko},
       title={{Takens' last problem and existence of non-trivial wandering
  domains}},
        date={2017},
     journal={Advances in Mathematics},
      volume={306},
       pages={524\ndash 588},
}

\bib{LR2016}{article}{
      author={Labouriau, Isabel~S},
      author={Rodrigues, Alexandre~AP},
       title={{On Takens' Last Problem: tangencies and time averages near
  heteroclinic networks}},
        date={2016},
     journal={arXiv preprint arXiv:1606.07017},
}

\bib{Ruelle2001}{article}{
      author={Ruelle, David},
       title={Historical behaviour in smooth dynamical systems},
        date={2001},
     journal={Global analysis of dynamical systems},
       pages={63\ndash 66},
}

\bib{Takens2008}{article}{
      author={Takens, Floris},
       title={Orbits with historic behaviour, or non-existence of averages},
        date={2008},
     journal={Nonlinearity},
      volume={21},
      number={3},
       pages={T33\ndash T36},
}

\bib{Viana99}{book}{
      author={Viana, Marcelo},
       title={Lecture notes on attractors and physical measures},
   publisher={IMCA},
        date={1999},
}

\end{biblist}
\end{bibdiv}

\end{document}